\documentclass{amsproc}

\usepackage{amssymb}

\usepackage{graphicx}

\usepackage[cmtip,all]{xy}

\usepackage[all]{xy}
\usepackage{amsthm}
\usepackage{amscd}
\usepackage{amsmath}
\usepackage{amsfonts}
\usepackage{amssymb}
\usepackage{stmaryrd}
\usepackage{graphicx}%
\usepackage{hyperref}
\hypersetup{
    bookmarks=true,         
    unicode=false,          
    pdftoolbar=true,        
    pdfmenubar=true,        
    pdffitwindow=false,     
    pdfstartview={FitH},    
    pdftitle={My title},    
    pdfauthor={Author},     
    pdfsubject={Subject},   
    pdfcreator={Creator},   
    pdfproducer={Producer}, 
    pdfkeywords={keyword1, key2, key3}, 
    pdfnewwindow=true,      
    colorlinks=false,       
    linkcolor=red,          
    citecolor=green,        
    filecolor=magenta,      
    urlcolor=cyan           
}

\newtheorem{theorem}{Theorem}[section]
\newtheorem{lemma}[theorem]{Lemma}

\theoremstyle{definition}
\newtheorem{definition}[theorem]{Definition}
\newtheorem{example}[theorem]{Example}
\newtheorem{proposition}[theorem]{Proposition}
\newtheorem{corollary}[theorem]{Corollary}

\theoremstyle{remark}
\newtheorem{remark}[theorem]{Remark}

\numberwithin{equation}{section}

\begin{document}

\title{Quasi-Elliptic Cohomology and its Spectrum}



\author{Zhen Huan}

\address{Zhen Huan, Department of Mathematics,
Sun Yat-sen University, Guangzhou, 510275 China} \curraddr{}
\email{huanzhen@mail.sysu.edu.cn}
\thanks{The author was partially supported by NSF grant DMS-1406121.}


\subjclass[2010]{Primary 55}

\date{}

\begin{abstract}
Ginzburg, Kapranov and Vasserot conjectured the existence of equivariant elliptic cohomology theories.
In this paper, to give a description of equivariant spectra of the theories, we study an intermediate theory, quasi-elliptic cohomology.
We formulate a new category of orthogonal $G-$spectra and construct explicitly an orthogonal $G-$spectrum of quasi-elliptic cohomology in it. The
idea of the  construction can be applied to a family of equivariant cohomology theories, including Tate K-theory and generalized
Morava E-theories. Moreover, this construction provides a functor from the category
of global spectra to the category of orthogonal $G-$spectra. In addition, from it we obtain some new idea what global homotopy theory is right for
constructing global elliptic cohomology theory.
\end{abstract}

\maketitle 
\section{Introduction}

An elliptic cohomology theory is an even periodic multiplicative
generalized cohomology theory whose associated formal group is the
formal completion of an elliptic curve. Elliptic cohomology theories serve as a family of algebraic variants reflecting
the geometric nature of elliptic curves, which make themselves intriguing and significant subjects to study.
One renowned conclusion on the representing spectra of the theories is Goerss-Hopkins-Miller theorem \cite{Lurie}.
It constructs many examples of
$E_{\infty}-$rings which represent elliptic cohomology theories,
including Tate K-theory.

Moreover, as K-theory and many other cohomology theories, elliptic cohomology  theories also have equivariant version.
In \cite{GKV},
Ginzburg, Kapranov and Vasserot gave the axiomatic definition of
$G-$equivariant elliptic cohomology theory.
They have the conjecture that any
elliptic curve $A$ gives rise to a unique equivariant elliptic
cohomology theory, natural in $A$.  In his thesis \cite{Gepnerthesis}, Gepner presented a
construction of the equivariant elliptic cohomology that satisfies
a derived version of the Ginzburg-Kapranov-Vasserot axioms.
We have the question from another perspective whether we can construct an orthogonal $G-$spectrum representing each
equivariant elliptic cohomology theory.

This question, however, is not easy to answer by studying elliptic cohomology theories themselves. They are intricate and mysterious theories.
Instead, we turn to an intermediate theory, quasi-elliptic cohomology theory. The idea of quasi-elliptic cohomology is  motivated by Ganter's
construction of Tate K-theory. Rezk established the theory in his unpublished manuscript \cite{Rez11}. The author gave a detailed description
of the construction of the theory in Chapter 2, \cite{Huanthesis} and Section 2, 3, \cite{huanqp}. Currently the author is writing a survey on quasi-elliptic cohomology\cite{Huansurvey}.

Quasi-elliptic cohomology theory is a variant of Tate K-theory, which is the generalized elliptic cohomology theory associated to
the Tate curve. The Tate curve $Tate(q)$ is an  elliptic curve
over Spec$\mathbb{Z}((q))$, which is classified as the completion
of the algebraic stack of some nice generalized elliptic curves at
infinity. A good reference for $Tate(q)$ is Section 2.6 of
\cite{AHS}. Tate K-theory itself is a distinctive subject to study. The
relation between Tate K-theory and string theory is better
understood than for most known elliptic cohomology theories. In addition, the
definition of $G-$equivariant Tate K-theory for finite groups $G$
is modelled on the loop space of a global quotient orbifold, which
is formulated explicitly in Section 2, \cite{Gan07}.

Quasi-elliptic cohomology theory contains all the information of Tate K-theory and reflects the geometric nature of elliptic curves. Moreover, it has many advantages \cite{huanqp}\cite{Huanthesis}.
One large good feature that you can tell is it can be expressed explicitly by equivariant K-theories.
\begin{equation}QEll^*_G(X):=\prod_{\sigma\in
G^{tors}_{conj}}K^*_{\Lambda(\sigma)}(X^{\sigma})=\bigg(\prod_{\sigma\in
G^{tors}}K^*_{\Lambda(\sigma)}(X^{\sigma})\bigg)^G.\label{defintro}\end{equation}
Equivariant K-theory is a classical example of equivariant cohomology theories. It has been thoroughly studied and has many good features.
Comparing with any elliptic cohomology theory, it is more practicable to construct the representing spectra of quasi-elliptic cohomology theory.

Then, how practicable is it? One immediate idea is that if we can construct a right adjoint functor $r_{\sigma}$ of each fixed point functor $X\mapsto X^{\sigma}$,
then a representing spectra $\{X_n, \psi_n\}_n$ of the theory $QEll^*_G(-)$ can be constructed by \begin{equation}X_n=\prod_{\sigma\in
G^{tors}_{conj}}r_{\sigma}(KU_{\Lambda_G(\sigma), n})\mbox{,   } \mbox{   } \psi_n=\prod_{\sigma\in
G^{tors}_{conj}}r_{\sigma}(\phi_n)\label{imagine1}\end{equation} where $\{KU_{G, n}, \phi_n\}_n$ denotes a $G-$spectrum representing $K_G^*(-)$.
However, the fixed point functor does not  have right adjoint.
Consequently, we introduce the concept of homotopical adjunction. Via homotopical right adjoints of fixed point functors, the representing spectrum
that we obtain stays in a new category of orthogonal $G-$spectra $GwS$. 

We are still studying whether this way of constructing  the orthogonal $G-$spectrum
 of quasi-elliptic cohomology can be applied to the elliptic cohomology theories
  and whether the category $GwS$ is the right category for equivariant elliptic spectra to reside at. But
this idea can be applied to a family of theories, including generalized
Morava E-theories and equivariant Tate K-theory. We can
construct in the category $GwS$ the orthogonal $G-$spectrum  of any theory of the form
\begin{equation}QE^*_G(X):=\prod_{\sigma\in
G^{tors}_{conj}}E^*_{\Lambda(\sigma)}(X^{\sigma})=\bigg(\prod_{\sigma\in
G^{tors}}E^*_{\Lambda(\sigma)}(X^{\sigma})\bigg)^G\label{defintroE}\end{equation}
with  $E$  any equivariant cohomology theory having the same key
features as equivariant K-theory, as explained in detail at the beginning of Section \ref{GorthospectraQEll}.

As equivariant K-theories, quasi-elliptic cohomology also has the
change-of-group isomorphism. In a conversation, Ganter indicated
that it has better chances than Grojnowski equivariant elliptic
cohomology theory to be put together naturally in a uniform way
and made into an ultra-commutative global cohomology theory in the
sense of Schwede \cite{SS}.

However, this orthogonal $G-$spectrum of quasi-elliptic cohomology cannot arise from an
orthogonal spectrum, i.e. this orthogonal $G-$spectrum is not the
underlying orthogonal $G-$spectrum of any orthogonal spectrum.
Instead, in  a coming paper we construct a new global homotopy
theory and  show there is a global orthogonal spectrum  in it that represents orthogonal
quasi-elliptic cohomology. Some construction and idea of this new theory has already been presented in Chapter 6 and 7, \cite{Huanthesis}.

It is worth mentioning that, other than Schwede's model for global
homotopy theory, there is a presheaf model for the theory shown in
\cite{Gepnerorbi}.  In \cite{Rez14} Rezk briefly introduced this
definition with differences in detail and he highlighted the role
of "cohesion" in relating ordinary equivariant homotopy theory
with global equivariant homotopy theory. This may be a better
model to construct global elliptic cohomology theories though the author
has not worked into it deeply.

\subsection{Where should we construct the equivariant spectrum?}

As indicated above, the equivariant spectrum of quasi-elliptic cohomology cannot be constructed as in (\ref{imagine1}) because the right
adjoint functor $r_{\sigma}$ does not exist.   
We generalize the concept of  right adjoints a little and introduce homotopical adjunction.

\begin{definition}[homotopical adjunction] Let $H$ and $G$ be two compact Lie groups.
Let \begin{equation} L: G\mathcal{T} \longrightarrow H\mathcal{T}
\mbox{ and } R: H\mathcal{T} \longrightarrow
G\mathcal{T}\end{equation} be two functors. A
\textit{left-to-right homotopical adjunction} is a natural map
\begin{equation}\mbox{Map}_H(LX,Y)\longrightarrow
\mbox{Map}_G(X,RY),\end{equation} which is a weak equivalence of
spaces when $X$ is a $G-$CW complex.

Analogously, a \textit{right-to-left homotopical adjunction} is a
natural map
\begin{equation}
\mbox{Map}_G(X,RY)\longrightarrow \mbox{Map}_H(LX,Y)\end{equation}
which is a weak equivalence of spaces when $X$ is a $G-$CW
complex.

$L$ is called a \textit{homotopical left adjoint} and $R$ a
\textit{homotopical right adjoint}. \label{holrintro}
\end{definition}

The homotopical right adjoint $\mathcal{R}_{\sigma}$ of the fixed point functor $X\mapsto X^{\sigma}$ exists. We give an explicit construction
of it in Theorem \ref{KUM}. Via these $\mathcal{R}_{\sigma}$s, we construct a $G-$space $QE_{G, n}$. Its relation with the theory $QE_{G}^n(-)$
is \begin{equation}\pi_0(QE_{G, n})=QE^n_G(S^0),\label{weaksenspintro}\end{equation} as shown in Theorem \ref{Emain}. This construction motivates us to construct
the category $GwT$ of $G$-spaces.  It is defined to be the homotopy category of the category of $G-$spaces with the weak equivalence defined by
\begin{equation}A\sim B\mbox{  if   }\pi_0(A)=\pi_0(B).\label{weaksenSintro}\end{equation}

Moreover, we can define the category $GwS$ of orthogonal $G-$spectra, which is the homotopy category of the category of orthogonal $G-$spectra with the weak equivalence defined by
\begin{equation}X\sim Y\mbox{  if   }\pi_0(X(V))=\pi_0(Y(V)),\label{weaksensp12}\end{equation} for each   faithful  $G-$representation  $V$.
And an orthogonal $G-$spectrum $X$ is said to represent a theory $H^*_G$ in $GwS$ if we have a natural map \begin{equation}\pi_0(X(V))=H^V_G(S^0),\label{weaksenspV}\end{equation} for each   faithful  $G-$representation  $V$.

\subsection{Orthogonal $G-$spectra}

To construct an orthogonal $G-$spectrum strictly representing $QE^*_G$ is far
beyond our imagination. But the construction of an orthogonal $G-$spectrum
representing it in $GwS$ is down-to-earth.

The readers may have noticed that we cannot construct the structure maps in the same way as  (\ref{imagine1}) to equip
the spaces $\{QE_{G, n}\}_n$ the structure of a $G-$spectrum because we are using homotopical right adjoints.
We have the same problem when trying to construct an orthogonal $G-$spectrum representing $QE^*_G$ in the category $GwS$.

But when $E_G$ is a $\mathcal{I}_G-$FSP,  we can construct the
structure maps for $QE(G, -)$ explicitly and obtain the main conclusion of
this paper.
\begin{theorem}
If the equivariant cohomology theory $E^*_G$ can be represented by
a $\mathcal{I}_G-$FSP $(E_G, \eta^{E}, \mu^{E})$,  there is a
$\mathcal{I}_G-$FSP
$$(QE(G, -), \eta^{QE}, \mu^{QE})$$ representing
$QE^*_G$  in $GwS$. In the case that $(E_G, \eta^{E}, \mu^{E})$ is
commutative, $(QE(G, -), \eta^{QE}, \mu^{QE})$ is commutative.
\label{mainmain}
\end{theorem}
The construction of $(QE(G, -), \eta^{QE}, \mu^{QE})$ gives us a functor.

\begin{theorem}
There is a well-defined functor $\mathcal{Q}$ from the full subcategory consisting of
$\mathcal{I}_G-$FSP in $GwS$ to the same category sending $(E_G, \eta^E, \mu^E)$  to
$(QE(G, -), \eta^{QE}, \mu^{QE})$.

The restriction of $\mathcal{Q}$ to the full subcategory consisting of
commutative $\mathcal{I}_G-$FSP  is a functor from that category  to itself.
 \end{theorem}

Moreover, we have the corollary for quasi-elliptic cohomology.

\begin{theorem}
There  is a commutative $\mathcal{I}_G-$FSP $(QK(G, -), \eta^{QK},
\mu^{QK})$ representing quasi-elliptic cohomology in $GwS$.
\label{mainmaincor}
\end{theorem}

In addition, we construct the restriction maps $QE(G,
V)\longrightarrow QE(H, V)$ for each group homomorphism
$H\longrightarrow G$. This map is not a homeomorphism, but an
$H-$weak equivalence.

As shown in Section \ref{newglobal}, the  orthogonal $G-$spectrum
$(QE(G, -), \eta^{QE}, \mu^{QE})$ cannot arise from an orthogonal
spectrum. This fact motivates us to construct a new global
homotopy theory in a coming paper \cite{Huanglobal}.

\bigskip
In Section \ref{bdg}  we recall the basics in equivariant homotopy
theory. In Section \ref{conqec} we recall the construction of
quasi-elliptic cohomology.  In Section \ref{newcat} we introduce homotopical adjunction and construct the category $GwS$
of orthogonal $G-$spaces. In Section \ref{basicweakconstr}, we construct a homotopical right adjoint of the fixed point functor
and then
show the construction of a space $E_{G, n}$ representing $E^n_G(-)$ in $GwT$. In Section
\ref{GorthospectraQEll} we construct an orthogonal $G-$spectrum
for quasi-elliptic cohomology, which is a commutative $\mathcal{I}_G-$FSP in $GwS$.
In Section \ref{gaction}, we define the restriction map of these equivariant orthogonal spectra.
In Section \ref{newglobal} we give a brief introduciton of ideas related to global homotopy theory.

In Appendix \ref{join}, we recall the definition and properties of join.
In Appendix \ref{recalortho} we recall the basics of
global homotopy theory and the construction of global K-theory.  In
Appendix \ref{reallambda} we construct some faithful
representations of the group $\Lambda_G(g)$ which are essential in
the construction of orthogonal $G-$spectrum. And we put the technical proofs of some conclusions in
Appendix \ref{tediousproof}.
\bigskip

\textit{Acknowledgement. } I thank Charles Rezk for suggesting the
idea of homotopical adjunction and weak spectra. He initiated the
project and supported my work all the time.  I also thank Matthew
Ando, Stefan Schwede, Nathaniel Stapleton and Guozhen Wang for
helpful remarks.

\section{Notations in equivariant homotopy theory}\label{bdg}

In this section  we give a sketch of the notations and conclusions
in the equivariant homotopy theory that we need in further
sections. The main references are \cite{Bredon},
\cite{Elmendorfthesis} and \cite{Alaska}.

Let $G$ be a compact Lie group. Let $\mathcal{T}$ denote the
category of topological spaces and continuous maps. Let
$G\mathcal{T}$ denote the category of $G-$spaces, namely, spaces
$X$ equipped with continuous $G-$action $G\times X\longrightarrow
X$ and continuous $G-$maps.

Let $H$ be a closed subgroup of $G$. Let $X$ be a $G-$space and
$Y$ an $H-$space. Define
\begin{equation}X^H:=\{x|hx=x,  \forall h\in H\}.\end{equation}
For $x\in X$, the isotropy group of $x$
\begin{equation}G_x : =\{h|hx=x\}.\end{equation}

Let $hG\mathcal{T}$ denote the homotopy category whose objects are
$G-$spaces and morphisms are $G-$homotopy classes of continuous
$G-$maps. Let $\overline{h}G\mathcal{T}$ denote the category
constructed from $hG\mathcal{T}$ by adjoining formal inverses to
the weak equivalences.

\begin{theorem}[Elemendorf's Theorem]

The category $\overline{h}\mathcal{T}^{\mathcal{O}_G^{op}}$ and
$\overline{h}G\mathcal{T} $ are
equivalent.\label{elem}\end{theorem} Let $G\mathcal{C}$ denote the
category of $G-$CW complexes and celluar maps.
Proposition \ref{keyl} is a conclusion needed for the construction later. 
It can be proved by induction over cells.
\begin{proposition}
Let $D$ be a complete category. Let $i:
\mathcal{O}^{op}_G\longrightarrow G\mathcal{C}^{op}$ be the
inclusion of subcategory. If $F_1, F_2:
G\mathcal{C}^{op}\longrightarrow D$ are two functors sending
homotopy colimits to homotopy limits and if we have a natural
transformation $p: F_1\longrightarrow F_2$, which gives a weak
equivalence at orbits, then it also gives a weak equivalence on
$G\mathcal{C}$. Especially, if $p$ gives a retract at each orbit,
$F_1$ is a retract of $F_2$ at each $G-$CW complexes. \label{keyl}
\end{proposition}

\section{Quasi-elliptic cohomology}\label{conqec}

In this section we recall the definition of quasi-elliptic
cohomology. The main reference is \cite{huanqp}, \cite{Huanthesis} and \cite{Rez11}.
Before that we discuss in Section \ref{lambdarepresentationlemma}
the representation ring of $\Lambda_G(g)$.

\subsection{Preliminary: representation ring of
$\Lambda_G(g)$}\label{lambdarepresentationlemma}

For any compact Lie group $G$ and a torsion element $g\in G$, let
$C_G(g)$ denote the centralizer of $g$ in $G$, and let
$\Lambda_G(g)$ denote the group
$$\Lambda_G(g)=C_G(g)\times\mathbb{R}/\langle(g, -1)\rangle.$$ Let $\mathbb{T}$ denote the circle group
$\mathbb{R}/\mathbb{Z}$. Let $q: \mathbb{T}\longrightarrow U(1)$
be the isomorphism $t\mapsto e^{2\pi it}$. The representation ring
$R\mathbb{T}$ is $\mathbb{Z}[q^{\pm}]$.

We have an exact sequence
$$1\longrightarrow C_G(g)\longrightarrow
\Lambda_G(g)\buildrel{\pi}\over\longrightarrow\mathbb{T}\longrightarrow
0$$ where the first map is $g\mapsto [g, 0]$ and the second  is
\begin{equation}\pi([g, t])= e^{2\pi it}.\label{pizq}\end{equation}

There is a relation between the representation ring of $C_G(g)$
and that of $\Lambda_G(g)$.
\begin{lemma}

$\pi^*: R\mathbb{T}\longrightarrow R\Lambda_G(g)$ exhibits
$R\Lambda_G(g)$ as a free $R\mathbb{T}-$module.

In particular, there is an $R\mathbb{T}-$basis of $R\Lambda_G(g)$
given by irreducible representations $\{V_{\lambda}\}$, such that
restriction $V_{\lambda}\mapsto V_{\lambda}|_{C_G(g)}$ to $C_G(g)$
defines a bijection between $\{V_{\lambda}\}$ and the set
$\{\lambda\}$ of irreducible representations of
$C_G(g)$.\label{cl}
\end{lemma}
The proof is in \cite{huanqp} and also \cite{Huanthesis}.

\begin{remark} We can make a canonical choice of $\mathbb{Z}[q^{\pm}]$-basis
for $R\Lambda_{G}(g)$. For each irreducible $G$-representation
$\rho: G\longrightarrow Aut(G)$, write $\rho(\sigma)=e^{2\pi
ic}id$ for $c\in[0,1)$, and set $\chi_{\rho}(t)=e^{2\pi ict}$.
Then the pair $(\rho, \chi_{\rho})$ corresponds to a unique
irreducible $\Lambda_{G}(g)$-representation
\begin{equation}\rho\odot_{\mathbb{C}} \chi_{\rho}([h, t]):
=\rho(h)\chi_{\rho}(t).\end{equation}
\label{lambdabasis}\end{remark}

\subsection{Quasi-elliptic cohomology}\label{orbqec} In this
section we introduce the definition of quasi-elliptic cohomology
$QEll^*_G$ in term of  equivariant K-theory. This theory can also
be constructed from Rezk's ghost loops defined in \cite{Rez16}. To
see a full discussion about the relation between equivariant loop
spaces and quasi-elliptic cohomology, please refer to Chapter 2 and 3
\cite{huanqp}, Chapter 2 \cite{Huanthesis} and \cite{Rez11}.

Let $X$ be a $G-$space. Let $G^{tors}\subseteq G$ be the set of
torsion elements of $G$. Let $\sigma\in G^{tors}$. The fixed point
space $X^{\sigma}$ is a $C_G(\sigma)-$space. We can define a
$\Lambda_G(\sigma)-$action on $X^{\sigma}$ by $[g, t]\cdot
x:=g\cdot x.$

\begin{definition} The quasi-elliptic cohomology is defined by
\begin{equation}QEll^*_G(X)=\prod_{g\in
G^{tors}_{conj}}K^*_{\Lambda_G(g)}(X^{g})=\bigg(\prod_{g\in
G^{tors}}K^*_{\Lambda_G(g)}(X^{g})\bigg)^G,\end{equation} where
$G^{tors}_{conj}$ is a set of representatives of $G-$conjugacy
classes in $G^{tors}$.\label{qelldef}
\end{definition}
We have the ring homomorphism
$\mathbb{Z}[q^{\pm}]=K^0_{\mathbb{T}}(\mbox{pt})\buildrel{\pi^*}\over\longrightarrow
K^0_{\Lambda_G(g)}(\mbox{pt})\longrightarrow
K^0_{\Lambda_G(g)}(X)$ where $\pi: \Lambda_G(g)\longrightarrow
\mathbb{T}$ is the projection defined in (\ref{pizq}) and the
second is via the collapsing map $X\longrightarrow \mbox{pt}$. So
$QEll_G^*(X)$ is naturally a
$\mathbb{Z}[q^{\pm}]-$algebra. 

Similar to equivariant K-theories, we can  construct the
restriction map, the K\"{u}nneth map on it, its tensor product and
the change-of-group isomorphism of quasi-elliptic cohomology. We
construct the restriction map and the change-of-group isomorphism
in this section. For other constructions, please refer to
\cite{huanqp}.

Since each homomorphism $\phi: G\longrightarrow H$ induces a
well-defined homomorphism $\phi_{\Lambda}:
\Lambda_G(\tau)\longrightarrow\Lambda_H(\phi(\tau))$ for each
$\tau$ in $G$, we can get the proposition below directly.
\begin{proposition}For each homomorphism $\phi: G\longrightarrow H$, it induces a ring map
$$\phi^*: QEll^*_H(X)\longrightarrow QEll^*_G(\phi^*X)$$ characterized by the commutative diagrams

\begin{equation}\begin{CD}QEll^*_H(X) @>{\phi^*}>> QEll^*_G(\phi^*X) \\ @V{\pi_{\phi(\tau)}}VV  @V{\pi_{\tau}}VV  \\
K^*_{\Lambda_H(\phi(\tau))}(X^{\phi(\tau)}) @>{\phi^*_{\Lambda}}>>
 K^*_{\Lambda_G(\tau)}(X^{\phi(\tau)})\end{CD}\end{equation} for any $\tau \in G$. So $QEll^*_G$ is functorial in $G$.
\label{restrictionq}\end{proposition}

We also have the change-of-group isomorphism as  in equivariant
$K$-theory.

Let $H$ be a subgroup of $G$ and $X$ an $H$-space. Let $\phi:
H\longrightarrow G$ denote the inclusion homomorphism. The
change-of-group map $\rho^G_H: QEll^*_G(G\times_HX)\longrightarrow
QEll^*_H(X)$ is defined as the composite
\begin{equation}\rho^G_H:   QEll^*_G(G\times_HX)\buildrel{\phi^*}\over\longrightarrow
QEll^*_H(G\times_H X)\buildrel{i^*}\over\longrightarrow
QEll_H^*(X)\label{changeofgroup}
\end{equation}
where $\phi^*$ is the restriction map and $i: X\longrightarrow
G\times_HX$ is the $H-$equivariant map defined by $i(x)=[e, x].$

\begin{proposition} The change-of-group map
$$\rho^G_H: QEll^*_G(G\times_H X)\longrightarrow
QEll^*_H(X)$$ defined in (\ref{changeofgroup}) is an
isomorphism.\end{proposition}
\begin{proof}
For any $\tau\in H_{conj}$, there exists a unique
$\sigma_{\tau}\in G_{conj}$ such that
$\tau=g_{\tau}\sigma_{\tau}g_{\tau}^{-1}$ for some $g_{\tau}\in
G$.  Consider the maps \begin{equation}\begin{CD}
\Lambda_G(\tau)\times_{\Lambda_H(\tau)}X^{\tau}@>{[[a, t], x
]\mapsto [a, x]}>> (G\times_H X)^{\tau}@>{[u, x]\mapsto
[g_{\tau}^{-1}u, x]}>>
(G\times_HX)^{\sigma}.\end{CD}\end{equation} The first map is
$\Lambda_G(\tau)-$equivariant and the second is equivariant with
respect to the homomorphism $c_{g_{\tau}}:
\Lambda_{G}(\sigma)\longrightarrow \Lambda_G(\tau)$ sending $[u,
t]\mapsto [g_{\tau} u g_{\tau}^{-1}, t]$. Taking a coproduct over
all the elements $\tau\in H_{conj}$ that are conjugate to
$\sigma\in G_{conj}$ in $G$, we get an isomorphism
$$\gamma_{\sigma}: \coprod_{\tau}\Lambda_G(\tau)\times_{\Lambda_H(\tau)} X^{\tau}\longrightarrow
(G\times_HX)^{\sigma}$$ which is $\Lambda_G(\sigma)-$equivariant
with respect to $c_{g_{\tau}}$. Then we have the map
\begin{equation}\gamma:=\prod_{\sigma\in G_{conj}}\gamma_{\sigma}:
\prod_{\sigma\in G_{conj}}
K^*_{\Lambda_G(\sigma)}(G\times_HX)^{\sigma}\longrightarrow
\prod_{\sigma\in
G_{conj}}K^*_{\Lambda_G(\sigma)}(\coprod_{\tau}\Lambda_G(\tau)\times_{\Lambda_H(\tau)}
X^{\tau})
\end{equation}

It is straightforward to check the change-of-group map coincide
with the composite \begin{align*} QEll^*_{G}(G\times_H
X)\buildrel{\gamma}\over\longrightarrow \prod_{\sigma\in
G_{conj}}K^*_{\Lambda_G(\sigma)}(\coprod_{\tau}\Lambda_G(\tau)\times_{\Lambda_H(\tau)}
X^{\tau})\longrightarrow &\prod_{\tau\in
H_{conj}}K^*_{\Lambda_H(\tau)}(X^{\tau})\\&=QEll^*_{H}(X)\end{align*}
with  the second map  the change-of-group isomorphism in
equivariant $K-$theory.
\end{proof}

\section{A new category of orthogonal $G-$spectra}\label{newcat}

To construct a concrete representing spectrum for elliptic
cohomology is a difficult goal to achieve. We consider constructing
a representing spectrum of quasi-elliptic cohomology first, which is
not easy to realize, either.

In this section we first construct a new category of orthogonal $G-$spectra where quasi-elliptic cohomology resides.

The quasi-elliptic cohomology, as defined in (\ref{qelldef}), has the form $$QEll^*_G(X)=\prod_{g\in
G^{tors}_{conj}}K^*_{\Lambda_G(g)}(X^{g})=\bigg(\prod_{g\in
G^{tors}}K^*_{\Lambda_G(g)}(X^{g})\bigg)^G.$$
If we could construct the right adjoint of the fixed point functor $X\mapsto X^g$ from the category of $G-$spaces to that of $\Lambda_G(g)-$spaces, we can construct the representing
spectrum of the theory afterwards. However, the fixed point functor does not preserve colimits, thus, does not have right adjoints. Instead, we consider a concept weaker than adjoints.

\begin{definition}[homotopical adjunction] Let $H$ and $G$ be two compact Lie groups.
Let \begin{equation} L: G\mathcal{T} \longrightarrow H\mathcal{T}
\mbox{ and } R: H\mathcal{T} \longrightarrow
G\mathcal{T}\end{equation} be two functors. A
\textit{left-to-right homotopical adjunction} is a natural map
\begin{equation}\mbox{Map}_H(LX,Y)\longrightarrow
\mbox{Map}_G(X,RY),\end{equation} which is a weak equivalence of
spaces when $X$ is a $G-$CW complex.

Analogously, a \textit{right-to-left homotopical adjunction} is a
natural map
\begin{equation}
\mbox{Map}_G(X,RY)\longrightarrow \mbox{Map}_H(LX,Y)\end{equation}
which is a weak equivalence of spaces when $X$ is a $G-$CW
complex.

$L$ is called a \textit{homotopical left adjoint} and $R$ a
\textit{homotopical right adjoint}. \label{holr}
\end{definition}

Homotopical adjunction is another way to describe the relation between $G-$equivariant homotopy theory and those equivariant homotopy theory for its closed subgroups.
This definition can be generalized to functors between categories other than $H\mathcal{T}$ and
$G\mathcal{T}$. Homotopical adjunction is a notion more ubiquitous in category theory than adjunctions.

\begin{example} Let $G=\mathbb{Z}/2\mathbb{Z}$ and $g$ be a generator of $G$. We want to find a homotopical right adjoint $R$  of the functor
$X\mapsto X^g$ from the category $G\mathcal{T}$ of $G$-spaces to
the category $\mathcal{T}$ of topological spaces.

Let $Y$ be a topological space. Suppose we have
$$\mbox{Map}(X^g, Y)\simeq \mbox{Map}_{G}(X,
RY).$$ $G$ has two subgroups, $e$ and $G$.
\begin{align*}RY^e&=\mbox{Map}_{G}(G/e, RY)\simeq \mbox{Map}((G/e)^g,
Y)\simeq \mbox{pt};\\  RY^G&=\mbox{Map}_{G}(G/G, RY)\simeq
\mbox{Map}((G/G)^g, Y)=Y.\end{align*} If $Y$ is the empty set,
$R\emptyset$ is $EG$. And generally for any $Y$, one choice of
$RY$ is the join $Y \ast EG$.

By Elmendorf's theorem \ref{elem}, the space $RY$ is unique up to
$G-$homotopy. By definition, the functor $R$  is a homotopical
right adjoint to the fixed point functor $X\mapsto X^g$.

\label{motivating2}
\end{example}

After we find a homotopical right adjoint $R_g$ of the fixed point functor $X\mapsto X^g$, we can construct a space $QEll_{G, n}$ representing
the $n-$th $G-$equivariant quasi-elliptic cohomology $QEll^n_G(-)$ up to the weak equivalence \begin{equation}\pi_0(QEll_{G, n})=QEll^n_G(S^0),\label{weaksenqe}\end{equation}
In other words, $QEll_{G, n}$ represents $QEll^n_G(-)$ in the category $GwT$ below. The explicit construction of $QEll_{G,n}$ can be found in Corollary \ref{KUmain}.
\begin{definition}The category $GwT$ is the homotopy category of the category of $G-$spaces with the weak equivalence defined by
\begin{equation}A\sim B\mbox{  if   }\pi_0(A)=\pi_0(B).\label{weaksenS}\end{equation}

A $G-$space $A$ in $GwT$ is said to represent $H^n_G$ if we have a natural map \begin{equation}\pi_0(A)=H^n_G(S^0).\label{weaksenspV}\end{equation} \label{gwtdef}
\end{definition}

Moreover, we can consider the category below of orthogonal $G-$spectra.
\begin{definition}The category $GwS$ is the homotopy category of the category of orthogonal $G-$spectra with the weak equivalence defined by
\begin{equation}X\sim Y\mbox{  if   }\pi_0(X(V))=\pi_0(Y(V)),\label{weaksensp12}\end{equation} for each   faithful  $G-$representation  $V$.

An orthogonal $G-$spectrum $X$ in $GwS$ is said to represent a theory $H^*_G$ if we have a natural map \begin{equation}\pi_0(X(V))=H^V_G(S^0),\label{weaksenspV}\end{equation} for each   faithful  $G-$representation  $V$. \label{gwsdef}
\end{definition}

The orthogonal $G-$spectrum representing quasi-elliptic cohomology in $GwS$ is constructed in Section \ref{GorthospectraQEll}.


\section{Equivariant spectra}\label{basicweakconstr}
Let $E^*_G(-)$ be a $G-$equivariant cohomology theory. Define
\begin{equation}QE^*_G(X):=\prod_{\sigma\in
G^{tors}_{conj}}E^*_{\Lambda(\sigma)}(X^{\sigma})=\bigg(\prod_{\sigma\in
G^{tors}}E^*_{\Lambda(\sigma)}(X^{\sigma})\bigg)^G.\label{defintroE2}\end{equation}
In this section, for each integer $n$, each compact Lie group $G$,
we construct a space $QE_{G, n}$ representing the $n-$th
$G-$equivariant $QE^n_G$ up to weak equivalence. 



The construction of the right homotopical adjoint in Theorem
\ref{main} needs the space $S_{G,g}$ below.
For any compact Lie group $G$, let 
 $\langle g\rangle$ denote the cyclic subgroup of $G$
generated by $g\in G^{tors}$ and $\ast$  denote the join. Let
$$S_{G,g}:=\mbox{Map}_{\langle g\rangle}(G, \ast_{K}E(\langle
g\rangle/K))$$ where $K$ goes over all the maximal subgroups of
$\langle g\rangle$ and $E(\langle g\rangle/K)$ is the universal
space of the cyclic group $\langle g\rangle/K$. The action of
$\langle g\rangle/K$ on $E(\langle g\rangle/K)$ is free. For this
space $S_{G,g}$, it is classified up to $G-$homotopy, as shown in
the following lemma.

\begin{lemma}

For any closed subgroup $H\leqslant G$, $S_{G, g}$ satisfies
\begin{equation}S_{G, g}^H\simeq\begin{cases}\mbox{pt},&\text{if for any $b\in G$,  $b^{-1}\langle
g\rangle b\nleq
H$;}\\
\emptyset,&\text{if there exists a $b\in G$ such that
$b^{-1}\langle g\rangle b \leqslant
H$.}\end{cases}\end{equation}\label{sf}

\end{lemma}

\begin{proof}
For any closed subgroup $H$ of $G$.
\begin{equation}S_{G,g}^H=\mbox{Map}_{\langle g\rangle}(G/H, \ast_{K}E(\langle
g\rangle/K))\end{equation} where $K$ goes over all the cyclic
groups $\langle g^m\rangle$ with $\frac{|g|}{m}$ a prime.

If there exists an $b\in G$ such that $b^{-1}\langle g\rangle b
\leqslant H$, it is equivalent to say that there exists points in
$G/H$ that can be fixed by $g$. But there are no points in
$\ast_{K}E(\langle g\rangle/K)$ that can be fixed by $g$. So there
is no $\langle g\rangle-$equivariant map from $G/H$ to
$\ast_{K}E(\langle g\rangle/K)$. In this case $S_{G,g}^H$ is
empty.

If for any $b\in G$, $b^{-1}\langle g\rangle b\nleq H$, it is
equivalent to say that there are no points in $G/H$ that can be
fixed by $g$. And for any subgroup $\langle g^m\rangle$ which is
not $\langle g\rangle$ itself, $(\ast_{K}E(\langle
g\rangle/K))^{\langle g^m\rangle}$ is the join of several
contractible spaces $E(\langle g\rangle/K)^{\langle g^m\rangle}$,
thus, contractible. So all the homotopy groups
$\pi_n((\ast_{K}E(\langle g\rangle/K))^{\langle g^m\rangle})$ are
trivial. For any $n\geq 1$ and any $\langle g\rangle-$equivariant
map $$f: (G/H)^n\longrightarrow \ast_{K}E(\langle
g\rangle/K)$$from the $n-$skeleton of $G/H$, the obstruction
cocycle is zero.

Then by equivariant obstruction theory, $f$ can be extended to the
$(n+1)-$cells of $G/H$, and any two extensions $f$ and $f'$ are
$\langle g\rangle-$homotopic.


So in this case  $S_{G,g}^H$ is contractible.
\end{proof}
Theorem \ref{main} is crucial to the construction of $QE_{G, n}$.
\begin{theorem}
Let $G$ be a compact Lie group and $g\in G^{tors}$. A homotopical
right adjoint of the functor $L_g: G\mathcal{T}\longrightarrow
C_G(g)\mathcal{T},\mbox{ } X\mapsto X^g$ is
\begin{equation}R_g:
C_G(g)\mathcal{T}\longrightarrow G\mathcal{T}\mbox{,   }
Y\mapsto\mbox{Map}_{C_G(g)}(G, Y\ast S_{C_G(g), g}).\end{equation}
\label{main}\end{theorem}
\begin{proof}
Let $H$ be any closed subgroup of $G$.

First we show given a $C_G(g)-$equivariant map $f:
(G/H)^g\longrightarrow Y$, it extends uniquely up to
$C_G(g)-$homotopy to a $C_G(g)-$equivariant map $\widetilde{f}:
G/H\longrightarrow Y\ast S_{C_G(g), g}$. $f$ can be viewed as a
map $(G/H)^g\longrightarrow Y\ast S_{C_G(g), g}$ by composing with
the inclusion of one end of the join
$$Y\longrightarrow Y\ast S_{C_G(g), g},\mbox{           } y\mapsto (1y, 0).$$

If $bH\in (G/H)^g$,    define $\widetilde{f}(bH): =f(bH).$

If $bH$ is not in $(G/H)^g$, its stabilizer group does not contain
$g$.  By Lemma \ref{sf}, for any subgroup $K$ of it,  $S_{C_G(g),
g}^K$ is contractible. So $(Y\ast S_{C_G(g), g})^{K}=Y^{K}\ast
S_{C_G(g), g}^{K}$ is contractible. In other words, if $K$ occurs
as the isotropy subgroup of a point outside $(G/H)^g$,
$\pi_n((Y\ast S_{C_G(g), g})^{K})$ is trivial. By equivariant
obstruction theory, $f$ can extend
 to a $C_G(g)-$equivariant map $\widetilde{f}:
G/H\longrightarrow Y\ast S_{C_G(g), g}$, and any two extensions
are $C_G(g)-$homotopy equivalent. In addition, $S_{C_G(g), g}^{g}$
is empty. So the image of the restriction of any map
$G/H\longrightarrow Y\ast S_{C_G(g), g}$ to the subspace $(G/H)^g$
is contained in the end $Y$ of the join.

Thus, $\mbox{Map}_{C_G(g)}((G/H)^g, Y)$ is weak equivalent to
$\mbox{Map}_{C_G(g)}(G/H, Y\ast S_{C_G(g), g})$.

Moreover, we have the equivalence by adjunciton
\begin{equation}
\mbox{Map}_{G}(G/H, \mbox{           }\mbox{Map}_{C_G(g)}(G, Y\ast
S_{C_G(g), g}))\cong \mbox{Map}_{C_G(g)}(G/H, \mbox{ } Y\ast
S_{C_G(g), g})\end{equation} So we get
\begin{equation}R_gY^H=\mbox{Map}_{G}(G/H, R_gY)\backsimeq
\mbox{Map}_{C_G(g)}((G/H)^g,
Y)\label{weakorbitbasic}\end{equation} Let $X$ be  of the homotopy
type of a $G-$CW complex. Let $X^k$ denote the $k-$skeleton of
$X$.  Consider the functors
$$\mbox{Map}_G(-, R_gY)\mbox{ and }\mbox{Map}_{C_G(g)}((-)^g, Y)$$ from
$G\mathcal{T}$ to $\mathcal{T}$. Both of them sends homotopy
colimit to  homotopy limit. In addition, we have a natural map
from $\mbox{Map}_G(-, R_gY)$ to $\mbox{Map}_{C_G(g)}((-)^g, Y)$ by
sending a $G-$map $F: X\longrightarrow R_gY$ to the composition
\begin{equation}X^g\buildrel{F^g}\over\longrightarrow (R_gY)^g\longrightarrow
Y^g\subseteq Y\label{fi}\end{equation} with the second map
$f\mapsto f(e)$. Note that for any $f\in (R_gY)^g$, $f(e)=(g\cdot
f)(e)= f(e g)=f(g)=g\cdot f(e)$ so $f(e)\in (Y\ast S_{C_G(g),
g})^g=Y^g$ and the second map is well-defined. It gives weak
equivalence on orbits, as shown in (\ref{weakorbitbasic}). Thus,
by Proposition \ref{keyl}, $R_g$ is a homotopical right adjoint of
$L$. \end{proof}


\begin{theorem}Let $G$ be a compact Lie group, $g\in G^{tors}$, and $Y$ a $\Lambda_G(g)-$space. The subgroup $\{[(1, t)]\in\Lambda_G(g)|
t\in\mathbb{R}\}$ of $\Lambda_G(g)$ is isomorphic to $\mathbb{R}$.
We use the same symbol $\mathbb{R}$ to denote it. Consider the
functor $\mathcal{L}_g: G\mathcal{T}\longrightarrow
\Lambda_G(g)\mathcal{T}, X\mapsto X^g$ where $\Lambda_G(g)$ acts
on $X^g$ by $[g, t]\cdot x=gx.$ The functor $\mathcal{R}_g:
\Lambda_G(g)\mathcal{T}\longrightarrow G\mathcal{T}$ with
\begin{equation}\mathcal{R}_gY=\mbox{Map}_{C_G(g)}(G, Y^{\mathbb{R}}\ast S_{C_G(g), g})\end{equation} is a homotopical right adjoint of $\mathcal{L}_g$.\label{KUM}
\end{theorem}

\begin{proof}
Let $X$ be a $G-$space. Let $H$ be any closed subgroup of $G$.
Note for any $G-$space $X$, $\mathbb{R}$ acts trivially on $X^g$,
thus, the image of any $\Lambda_G(g)-$equivariant map
$X^g\longrightarrow Y$ is in $Y^{\mathbb{R}}$. So we have
$\mbox{Map}_{\Lambda_G(g)}(X^g, Y)=\mbox{Map}_{C_G(g)}(X^g,
Y^{\mathbb{R}}).$

First we show $f: (G/H)^g\longrightarrow Y^{\mathbb{R}}$ extends
uniquely up to $C_G(g)-$homotopy to a $C_G(g)-$equivariant map
$\widetilde{f}: G/H\longrightarrow Y^{\mathbb{R}}\ast S_{C_G(g),
g}$. $f$ can be viewed as a map $(G/H)^g\longrightarrow
Y^{\mathbb{R}}\ast S_{C_G(g), g}$ by composing with the inclusion
as the end of the join
$$Y^{\mathbb{R}}\longrightarrow Y^{\mathbb{R}}\ast S_{C_G(g), g},\mbox{    } y\mapsto (1y, 0).$$

The rest of the proof is analogous to that of Theorem
\ref{main}.\end{proof}

Theorem \ref{KUM} implies Theorem \ref{Emain} directly.
\begin{theorem}
For any compact Lie group $G$ and any integer $n$, let $E_{G, n}$
denote the space representing the $n-$th $G-$equivariant
$E-$theory. Then each $QE^n_G(-)$ is represented  by the
space
$$QE_{G,
n}:=\prod_{g\in G^{tors}_{conj}}\mathcal{R}_g(KU_{\Lambda_G(g),
n})
$$ in the category $GwT$ where
$\mathcal{R}_g(E_{\Lambda_G(g), n})$ is the space
$$\mbox{Map}_{C_G(g)}(G, E_{\Lambda_G(g), n}^{\mathbb{R}}\ast
S_{C_G(g), g}).$$

\label{Emain}
\end{theorem}

And we have the corresponding conclusion for quasi-elliptic
cohomology.
\begin{corollary}
For any compact Lie group $G$ and any integer $n$, let $KU_{G, n}$
denote the space representing the $n-$th $G-$equivariant
$KU-$theory. The $n-$th
 quasi-elliptic cohomology
$QEll^n_G(-)$ is  represented  by the space
$$QEll_{G,
n}:=\prod_{g\in G^{tors}_{conj}}\mathcal{R}_g(KU_{\Lambda_G(g),
n})
$$ in the category $GwS$ where $\mathcal{R}_g(KU_{\Lambda_G(g), n})$ is
the space
$$\mbox{Map}_{C_G(g)}(G, KU_{\Lambda_G(g), n}^{\mathbb{R}}\ast
S_{C_G(g), g}).$$ \label{KUmain}
\end{corollary}

The construction of the orthogonal $G-$spectrum of $QE-$theory in
Section \ref{anotherweaksp} is based on that of $QE_{G, n}$.

\section{Orthogonal $G-$spectrum of $QE^*_G$}\label{GorthospectraQEll}
In this section, we consider equivariant cohomology
theories $E$ that can be represented by $\mathcal{I}_G-$FSP $(E_G,
\eta^{E}, \mu^{E})$ and have the same key features as equivariant
complex K-theories. More explicitly, \\ $\bullet$ The theories
$\{E_G^*\}_G$ have the change-of-group isomorphism, i.e. for any
closed subgroup $H$ of $G$ and $H-$space $X$, the change-of-group
map $\rho^G_H: E^*_G(G\times_HX)\longrightarrow E^*_H(X)$ defined
by $E^*_G(G\times_HX)\buildrel{\phi^*}\over\longrightarrow
E^*_H(G\times_H X)\buildrel{i^*}\over\longrightarrow E_H^*(X)$ is
an isomorphism where $\phi^*$ is the restriction map and $i:
X\longrightarrow
G\times_HX$ is the $H-$equivariant map defined by $i(x)=[e, x].$\\
$\bullet$ There exists an orthogonal spectrum $E$ such that for
any compact Lie group $G$ and "large" real $G-$representation $V$
and a compact $G-$space $B$ we have a bijection
$E_G(B)\longrightarrow [B_+, E(V)]^G$. And $(E_G, \eta^{E}, \mu^{E})$ is the underlying orthogonal $G-$spectrum of $E$.\\
$\bullet$ Let $G$ be a compact Lie group and $V$ an orthogonal
$G-$representation. For every ample $G-$representation $W$, the
adjoint structure map $\widetilde{\sigma}^E_{V, W}:
E(V)\longrightarrow \mbox{Map}(S^W, E(V\oplus W))$ is a $G-$weak
equivalence.

In this section based
on the spaces we construct in Section \ref{basicweakconstr}, we
construct a $\mathcal{I}_G-$FSP representing the theory
$QE$ in the category $GwS$ defined in Definition \ref{gwsdef}.

\subsection{The construction of $QE(G, -)$}\label{anotherweaksp}

Let $G$ be any compact Lie group.  In this section  we consider
the case that the equivariant cohomology theory $E$ can be
represented by a global spectrum $(E, \eta^{E}, \mu^{E})$ and show in Section \ref{qegcon}
that there is a $\mathcal{I}_G-$FSP $(QE(G, -), \eta^{QE}, \mu^{QE})$
representing $QE^V_G(-)$ in the category $GwS$. Before that we construct each ingredient in the construction.

\subsubsection{The construction of $S(G, V)_g$}First we construct an orthogonal version $S(G, V)_g:=
Sym(V)\setminus Sym(V)^g$ of the space $S_{G, g}$. It is the space
classified by the condition (\ref{sfnew}) which is also the  condition classifying $S_{G,g}$.

Let $g\in G^{tors}$ and $V$ a real $G-$representation. Let
$Sym^n(V)$ denote the $n-$th symmetric power  $V^{\otimes n} $,
which has an evident $G\wr\Sigma_n-$action on it. Let
$$Sym(V):=\bigoplus_{n\geq 0} Sym^n(V).$$

When $V$ is an ample $G-$representation, $Sym(V)$ is a
$G-$representation containing all the irreducible
$G-$representations.  Since in this case $V$ is faithful $G-$representation,
for any closed subgroup $H$ of $G$, $Sym(V)$ is a faithful
$H-$representation, thus, a complete $H-$universe.

We use $S(G, V)_g$ to denote the space $Sym(V)\setminus Sym(V)^g.$ The complex conjugation on $V$ induces an involution on it. Note that for any
subgroup $H$ of $G$ containing $g$, $S(H, V)_g$ has the same
underlying space as $S(G, V)_g$.
\begin{proposition}
Let $V$ be an orthogonal $G-$representation. For any closed
subgroup $H\leqslant C_G(g)$, $S(G, V)_g$ satisfies
\begin{equation}S(G, V)_g^H\simeq\begin{cases}\mbox{pt},&\text{if $\langle
g\rangle \nleq
H$;}\\
\emptyset,&\text{if $\langle g\rangle\leqslant
H$.}\end{cases}\end{equation}\label{sfnew}
\end{proposition}
\begin{proof}
If $\langle g\rangle\leqslant H$, $Sym(V)^H$ is a subspace of
$Sym(V)^g$, so $(Sym(V)\setminus Sym(V)^g)^H$ is empty. If
$\langle g\rangle \nleq H$, $g$ is not in $H$. To simplify the
symbol, let $Sym^{n, \perp}$ denote the orthogonal complement of
$Sym^n(V)^g$ in $Sym^n(V)$.
\begin{align*}(Sym(V)\setminus Sym(V)^g)^H& =
colim_{n\longrightarrow\infty}Sym^n(V)^H\setminus (Sym^n(V)^g)^H\\
&= colim_{n\longrightarrow\infty} (Sym^n(V)^g)^H\times
\big((Sym^{n, \perp})^H\setminus\{0\} \big)\end{align*} Let $k_n$
denote the dimension of $(Sym^{n, \perp})^H$. Then $ (Sym^{n,
\perp})^H\setminus\{0\}\backsimeq S^{k_n-1}.$  As $n$ goes to
infinity, $k_n$ goes to infinity. When $k_n$ is large enough,
$S^{k_n-1}$ is contractible. So $(Sym(V)\setminus Sym(V)^g)^H$ is
contractible.
\end{proof}

\subsubsection{The construction of $F_g(G, V)$}
Next, we construct a space $F_g(G,
V)$ representing the theory $E^{V^g}_{\Lambda_G(g)}(-)$.

The faithful real $\Lambda_G(g)-$representation
constructed in Section \ref{realdef} serves as an essential component of the construction.
If $V$ is a faithful $G-$representation, 
by Proposition \ref{farithreallambda}, we have the
faithful $\Lambda_G(g)-$representation $(V)^{\mathbb{R}}_g$. In addition, $V^g$ can be
considered as a
$\Lambda_G(g)-$representation with trivial $\mathbb{R}-$action.
 The space $E((V)^{\mathbb{R}}_g\oplus V^g)$ represents
$E^{(V)^{\mathbb{R}}_g\oplus V^g}_{\Lambda_G(g)}(-)$. So we have
$$\mbox{Map}(S^{(V)^{\mathbb{R}}_g}, E((V)^{\mathbb{R}}_g\oplus
V^g))$$  represents $E^{V^g}_{\Lambda_G(g)}(-)$ since $$[X^g,
\mbox{Map}(S^{(V)^{\mathbb{R}}_g}, E((V)^{\mathbb{R}}_g\oplus
V^g))]^{\Lambda_G(g)} $$ is isomorphic to $$[X^g\wedge S^{(V)^{\mathbb{R}}_g},
E((V)^{\mathbb{R}}_g\oplus V^g)]^{\Lambda_G(g)}=
E^{(V)^{\mathbb{R}}_g\oplus V^g}_{\Lambda_G(g)}(X^g\wedge
S^{(V)^{\mathbb{R}}_g})= E^{V^g}_{\Lambda_G(g)}(X^g).$$

To simplify the symbol, we use $F_g(G, V)$ to denote the space
$$\mbox{Map}_{\mathbb{R}}(S^{(V)^{\mathbb{R}}_g},
E((V)^{\mathbb{R}}_g\oplus V^g)).$$ Its basepoint $c_0$ is the
constant map to the basepoint of
$E((V)^{\mathbb{R}}_g\oplus V^g)$.

 $F_g: (G, V)\mapsto F_g(G, V)$
provides a functor from $\mathcal{I}_{G}$ to the category
$C_G(g)\mathcal{T}$ of $C_G(g)-$spaces. It has the properties below.
\begin{proposition}
Let $G$ and $H$ be compact Lie groups. Let $V$ be a real
$G-$representation and $W$ a real $H-$representation. Let $g\in
G^{tors}$, $h\in H^{tors}$.\\ (i)  We have the unit
map $\eta_g(G, V): S^{V^g}\longrightarrow F_g(G, V)$ and the
multiplication
$$\mu^F_{(g, h)}((G, V), (H, W)): F_g(G, V)\wedge F_h(H,
W)\longrightarrow F_{(g, h)}(G\times H, V\oplus W)$$  making the
unit, associativity and centrality of unit diagram  commute. And
$\eta_g(G, V)$ is $C_G(g)-$equivariant and $\mu^F_{(g, h)}((G, V),
(H, W))$ is $C_{G\times H}(g, h)-$ equivariant. \\ (ii)Let
$\Delta_G$ denote the diagonal map $G\longrightarrow G\times
G,\mbox{      }g\mapsto (g, g).$ Let $\widetilde{\sigma}_g(G, V,
W): F_g(G, V)\longrightarrow \mbox{Map}(S^{W^g}, F_g(G, V\oplus
W))$ denote the map
$$x\mapsto (w\mapsto \big(\Delta^*_{G}\circ \mu^F_{(g, h)}((G, V), (G, W))\big)\big(x, \eta_g(G, W)(w)\big)).$$ Then $\widetilde{\sigma}_g(G, V, W)$
is a $\Lambda_G(g)-$weak equivalence when $V$ is an ample
$G-$representation. \\ (iii) If $(E, \eta^{E}, \mu^{E})$ is
commutative, we have
\begin{equation}\mu^F_{(g, h)}((G, V), (H, W))(x\wedge y)=
\mu^F_{(h, g)}((H, W), (G, V))(y\wedge x)
\label{quasicommFg}\end{equation} for any $x\in F_g(G, V)$ and
$y\in F_h(H, W)$.

\label{newF}
\end{proposition} 


\begin{proof}
(i) Let $V_1$ and $V_2$ be orthogonal $G-$representations and $f:
V_1\longrightarrow V_2$ be a linear isometric isomorphism. $f$
gives the linear isometric isomorphisms 
$f_1: (V_1)^{\mathbb{R}}_g\longrightarrow (V_2)^{\mathbb{R}}_g$,
and $f_2: (V_1)^{\mathbb{R}}_g\oplus V_1^g\longrightarrow
(V_2)^{\mathbb{R}}_g\oplus V_2^g$. Then define $F_g(f):
F_g(V_1)\longrightarrow F_g(V_2)$ in this way: for any
$\mathbb{R}-$equivariant map $\alpha:
S^{(V_1)^{\mathbb{R}}_g}\longrightarrow
E((V_1)^{\mathbb{R}}_g\oplus V_1^g)$, $F_g(f)(\alpha)$ is the
composition
\begin{equation}S^{(V_2)^{\mathbb{R}}_g}\buildrel{S(f_1^{-1})}\over\longrightarrow  S^{(V_1)^{\mathbb{R}}_g}\buildrel{\alpha}\over\longrightarrow E((V_1)^{\mathbb{R}}_g\oplus V_1^g)
\buildrel{E(f_2)}\over\longrightarrow E((V_2)^{\mathbb{R}}_g\oplus
V_2^g)\end{equation} which is still $\mathbb{R}-$equivariant. It
is straightforward to check $F_g(Id)$ is the identity map, and for
morphisms $V_1\buildrel{f}\over\longrightarrow
V_2\buildrel{f'}\over\longrightarrow  V_3$ in $\mathcal{I}_G$, we
have $F_g(f'\circ f)=F_g(f')\circ F_g(f).$

(ii) Define the unit map $\eta_g(G, V): S^{V^g}\longrightarrow
F_g(G, V)$ by \begin{equation}v\mapsto (v'\mapsto
\eta^E_{(V)^{\mathbb{R}}_g\oplus V^g}(v\wedge v'))
\end{equation} where $\eta^E_{(V)^{\mathbb{R}}_g\oplus V^g}:
S^{(V)^{\mathbb{R}}_g\oplus V^g}\longrightarrow
E((V)^{\mathbb{R}}_g\oplus V^g)$ is the unit map for global
$E$-theory. Since $(V)^{\mathbb{R}}_g\oplus V^g$ is a
$\Lambda_G(g)-$representation, $\eta^E_{(V)^{\mathbb{R}}_g\oplus
V^g}$ is $\Lambda_G(g)-$equivariant. So $\eta_g(G, V)$ is
well-defined and $\Lambda_G(g)-$equivariant.

Define the multiplication $\mu^F_{(g, h)}((G, V), (H, W)): F_g(G,
V)\wedge F_h(H, W)\longrightarrow F_{(g, h)}(G\times H, V\oplus
W)$ by \begin{equation}\alpha\wedge\beta\mapsto (v\wedge w\mapsto
\mu^E_{V, W}(\alpha(v)\wedge\beta(w)))\label{muFKg}\end{equation}
where $\mu^E_{V, W}$ is the multiplication for global $E-$theory.
Since $\mu^E_{V, W}$ is $\Lambda_{G}(g)\times
\Lambda_{H}(h)-$equivariant, $\mu^F_{(g, h)}((G, V), (H, W))$ is
$C_{G\times H}(g, h)-$equivariant. It is straightforward to check
the unit map and multiplication make the unit, associativity and
centrality of unit diagram commute.

(iii) Since $V$ is a faithful $G-$representation, by Proposition
\ref{farithlambda}, $(V)^{\mathbb{R}}_g\oplus V^g$ is a faithful
$\Lambda_G(g)-$representation.  By Theorem \ref{KUGweak}, we have
the $\Lambda_G(g)-$weak equivalence $E((V)^{\mathbb{R}}_g\oplus
V^g)\buildrel{\widetilde{\sigma}^E}\over\longrightarrow
\mbox{Map}(S^{(W)^{\mathbb{R}}_g\oplus W^g}, E((V\oplus
W)^{\mathbb{R}}_g\oplus (V\oplus W)^g))$ where
$\widetilde{\sigma}^E$ is the right adjoint of the structure map
of $E$. Thus we have the $\Lambda_G(g)-$weak equivalence
\begin{align*}\mbox{Map}(S^{(V)^{\mathbb{R}}_g}, E((V)^{\mathbb{R}}_g\oplus
V^g))&\longrightarrow \mbox{Map}(S^{(V)^{\mathbb{R}}_g},
\mbox{Map}(S^{(W)^{\mathbb{R}}_g\oplus W^g}, E((V\oplus
W)^{\mathbb{R}}_g\oplus (V\oplus W)^g)))\\ &=\mbox{Map}(S^{W^g},
\mbox{Map}(S^{(V\oplus W)^{\mathbb{R}}_g}, E((V\oplus
W)^{\mathbb{R}}_g\oplus (V\oplus W)^g))),\end{align*} i.e. $F_g(G,
V)\backsimeq_{C_G(g)}\mbox{Map}(S^{W^g}, F_g(G, V\oplus W)).$

\bigskip

(iv) (\ref{quasicommFg}) comes directly from the commutativity of
 $E$.
\end{proof}

\subsubsection{The construction of $QE(G, V)$}
\label{qegcon}Recall in Theorem \ref{Emain} we  construct a $G-$space
$QE_{G, n}$ representing $QE^n_G(-)$ in $GwT$. With $F_g(G, V)$ and $S(G,
V)_g$ we can go further than that.

Apply Theorem \ref{KUM}, we get the conclusion below.
\begin{proposition}
Let $V$ be a faithful orthogonal $G-$representation. Let $B'(G,
V)$ denote the space $$\prod_{g\in
G^{tors}_{conj}}\mbox{Map}_{C_G(g)}(G, F_g(G, V)\ast S(G, V)_g).$$
$QE^V_G(-)$ is  represented by $B'(G, V)$
in $GwT$. \label{quickEFS}
\end{proposition}
The proof of Proposition \ref{quickEFS} is analogous to that of
Theorem \ref{KUmain} step by step.

One disadvantage of $\{B'(G,
V)\}_V$ is that it is not easy to see whether we can construct the structure maps to make it an orthogonal $G-$spectrum.
Instead, we consider the $G-$weak
equivalent spaces $\{QE(G, V)\}_V$ in Proposition \ref{QEinal}.

Below is the main theorem in Section \ref{anotherweaksp}. We will
use formal linear combination $$t_1a+t_2b\mbox{ with }0\leqslant
t_1, t_2\leqslant 1, t_1+t_2=1$$ to denote points in join, as
talked in Appendix \ref{join}.

\begin{proposition}Let $QE_g(G, V)$ denote $$\{t_1a+t_2b\in F_g(G, V)\ast S(G, V)_g| \|b\|\leqslant t_2\}/\{t_1c_0+t_2
b\}.
$$ It is the quotient space of a closed subspace of  the join $F_g(G, V)\ast S(G,
V)_g$ with all the points of the form $t_1c_0+t_2 b$ collapsed to
one point, which we pick as the basepoint of $QE_g(G, V)$, where
$c_0$ is the basepoint of $F_g(G, V)$. $QE_g(G, V)$ has the
evident $C_G(g)-$action. And it is $C_G(g)-$weak equivalent to
$F_g(G, V)\ast S(G, V)_g$. As a result, $\prod\limits_{g\in
G^{tors}_{conj}}\mbox{Map}_{C_G(g)}(G, QE_g(G, V))$ is $G-$weak
equivalent to $B'(G,
V)$. So
\begin{equation}QE(G, V):=\prod\limits_{g\in
G^{tors}_{conj}}\mbox{Map}_{C_G(g)}(G, QE_g(G, V))
\label{EGVnog}\end{equation}  represents $QE^V_G(-)$ in the category $GwT$. \label{QEinal}
\end{proposition}
\begin{proof}
First we show $F_g(G, V)\ast S(G, V)_g$ is $C_G(g)-$homotopy
equivalent to $$QE'_g(G, V):=\{t_1a+t_2b\in F_g(G, V)\ast S(G,
V)_g| \|b\|\leqslant t_2\}.$$

Note that $b\in S(G, V)_g$ is never zero. Let $j: QE'_g(G,
V)\longrightarrow F_g(G, V)\ast S(G, V)_g$ be the inclusion. Let
$p: F_g(G, V)\ast S(G, V)_g\longrightarrow QE'_g(G, V)$ be the
$C_G(g)-$map sending $t_1a+t_2b$ to $t_1a+t_2\frac{min\{\|b\|,
t_2\}}{\|b\|}b$. Both $j$ and $p$ are both continuous and
$C_G(g)-$equivariant. $p\circ j$ is the identity map of $QE'_g(G,
V)$. We can define a $C_G(g)-$homotopy
$$H: (F_g(G, V)\ast S(G, V)_g)\times I\longrightarrow F_g(G,
V)\ast S(G, V)_g$$ from the identity map on $F_g(G, V)\ast S(G,
V)_g$ to $j\circ p$ by shrinking. For any $t_1a+t_2b\in F_g(G,
V)\ast S(G, V)_g$, Define
\begin{equation}H(t_1a+t_2b, t):= t_1a+t_2((1-t)b+t\frac{min\{\|b\|, t_2\}}{\|b\|}b). \end{equation}

Then we show $QE'_g(G, V)$ is $G-$weak equivalent to $QE_g(G, V)$.
Let $q: QE'_g(G, V)\longrightarrow QE_g(G, V)$ be the quotient
map. Let $H$ be a closed subgroup of $C_G(g)$.

If $g$ is in $H$, since $S(G, V)_g^H$ is empty, so $QE_g(G, V)^H$
is in the end $F_g(G, V)$ and can be identified with $F_g(G,
V)^H$. In this case
 $q^H$ is the identity map.

If $g$ is not in $H$, $QE'_g(G, V)^H$ is contractible. The cone
$\{c_0\}\ast S(G, V)_g^H$ is contractible, so $q\big((\{c_0\}\ast
S(G, V)_g)^H\big)=q(\{c_0\}\ast S(G, V)_g^H)$ is contractible.
Note that the subspace of all the points of the form $t_1c_0+t_2b$
for any $t_1$
 and $b$  is $q\big((\{c_0\}\ast S(G, V)_g)^H\big)$. Therefore, $QE_g(G, V)^H=QE'_g(G, V)^H/q(\{c_0\}\ast S(G,
 V)_g)^H$ is contractible.

Therefore, $QE'_g(G, V)$ is $G-$weak equivalent to $F_g(G, V)\ast
S(G, V)_g$.
\end{proof}

Moreover, generalizing the construction in Proposition \ref{QEinal}, we have the conclusion below on homotopical right adjoints.
\begin{proposition}
Let $g\in G^{tors}$. Let $Y$ be a based $\Lambda_G(g)-$space.  Let
$\widetilde{Y}_g$ denote the $C_G(g)-$space
$$\{t_1a+t_2b\in Y^{\mathbb{R}}\ast S(G, V)_g| \|b\|\leqslant
t_2\}/\{t_1y_0+t_2 b\}.
$$ It is the quotient space of a closed subspace of $Y^{\mathbb{R}}\ast S(G,
V)_g$ with all the points of the form     $t_1y_0+t_2 b$
  collapsed to one point, i.e the basepoint of $\widetilde{Y}_g$,
where $y_0$ is the basepoint of $Y$. $\widetilde{Y}_g$ is
$C_G(g)-$weak equivalent to $Y^{\mathbb{R}}\ast S(G, V)_g$. As a
result, the functor $R_g: C_G(g)\mathcal{T}\longrightarrow
G\mathcal{T}$ with $R_g\widetilde{Y}=\mbox{Map}_{C_G(g)}(G,
\widetilde{Y}_g)$ is a homotopical right adjoint of $L:
G\mathcal{T}\longrightarrow C_G(g)\mathcal{T}\mbox{, }X\mapsto
X^g$.

\label{generalYS}
\end{proposition}

The proof 
is analogous to that of
Theorem \ref{KUM}.

\begin{remark}
We can consider $QE_g(G, V)$  as a quotient space of a subspace of
$F_g(G, V)\times Sym(V)\times I$  \begin{equation}\{(a, b, t)\in
F_g(G, V)\times Sym(V)\times I|  \|b\|\leqslant t; \mbox{   and }
b\in S(G, V)_g\mbox{   if   }t\neq
0\}\label{qimiaoanoano}\end{equation} by identifying points $(a,
b, 1)$ with $(a', b, 1)$, and collapsing all the points $(c_0, b,
t)$ for any $b$ and $t$. In other words, the end $F_g(G, V)$ in
the join $F_g(G, V)\ast S(G, V)_g$ is identified with the points
of the form $(a, 0, 0)$ in (\ref{qimiaoanoano}).
\label{ananotherEll}
\end{remark}

\begin{proposition}For each $g\in G^{tors}$, $$QE_g: \mathcal{I}_G\longrightarrow C_G(g)\mathcal{T},\mbox{   } (G, V)\mapsto QE_g(G, V)$$ is a well-defined functor.
As a result, $$QE: \mathcal{I}_G\longrightarrow
G\mathcal{T},\mbox{ } (G, V)\mapsto \prod\limits_{g\in
G^{tors}_{conj}}\mbox{Map}_{C_G(g)}(G, QE_g(G, V))$$ is a
well-defined functor. \end{proposition}
\begin{proof}
Let $V$ and $W$ be $G-$ representations and $f: V\longrightarrow
W$ a linear isometric isomorphism. Then $f$ induces a
$C_G(g)-$homeomorphism $F_g(f)$ from $F_g(G, V)$ to $F_g(G, W)$
and a $C_G(g)-$homeomorphism $S_g(f)$ from $S(G,V)_g$  to
$S(G,W)_g$. We have the well-defined map
$$QE_g(f): QE_g(G, V)\longrightarrow QE_g(G, W)$$ sending a point
represented by $t_1a+t_2b$ in the join to that represented by
$t_1F_g(f)(a)+t_2S_g(f)(b)$. And $QE(f): QE(G, V)\longrightarrow
QE(G, W)$ is defined by
$$\prod\limits_{g\in G^{tors}_{conj}} \alpha_g\mapsto\prod\limits_{g\in G^{tors}_{conj}} QE_g(f)\circ\alpha_g.$$

It is straightforward to check that all the axioms hold.
\end{proof}

\subsection{Construction of $\eta^{QE}$ and $\mu^{QE}$}\label{Gorthospstructure}
In this section we construct a unit map $\eta^{QE}$ and a
multiplication $\mu^{QE}$ so that we get a commutative
$\mathcal{I}_G-$FSP representing the $QE-$theory in $GwS$.

Let $G$ and $H$ be compact Lie groups, $V$ an orthogonal
$G-$representation and $W$ an orthogonal $H-$representation. We
use $x_g$ to denote the basepoint of $QE_g(G, V)$, which is
defined in Proposition \ref{QEinal}. Let $g\in G^{tors}$. For
each $v\in S^V$, there are $v_1\in S^{V^g}$ and $v_2\in
S^{(V^g)^{\perp}}$ such that $v=v_1\wedge v_2$. Let
$\eta^{QE}_g(G, V): S^V\longrightarrow QE_g(G, V)$ be the map
\begin{equation}\eta^{QE}_g(G,
V)(v):=\begin{cases}(1-\|v_2\|)\eta_g(G, V)(v_1)+ \|v_2\|v_2,
&\text{if $\|v_2\|\leqslant 1$;}\\ x_g, &\text{if
$\|v_2\|\geqslant 1$.}
\end{cases}\label{finaleta}\end{equation} 
\begin{lemma}The map $\eta^{QE}_g(G, V)$ defined in (\ref{finaleta}) is well-defined, continuous and $C_G(g)-$equivariant.
\label{etaEcontinuous}\end{lemma} The proof of Lemma
\ref{etaEcontinuous} is in Appendix \ref{etaEcontinuousProof}.

\begin{remark}
For any $g\in G^{tors}$, it's straightforward to check the diagram
below commutes.
$$\begin{CD}S^{V^g} @>\eta_g(G, V)>> F_g(G, V) \\
@VVV      @VVV \\
S^V @>\eta^{QE}_g(G, V)>> QE_g(G, V)\end{CD}$$ where both vertical
maps are inclusions. By Lemma \ref{etaEcontinuous}, the map
\begin{equation}\eta^{QE}(G, V): S^V\longrightarrow
\prod\limits_{g\in G^{tors}_{conj}}\mbox{Map}_{C_G(g)}(G, QE_g(G,
V))\mbox{,   } v \mapsto \prod\limits_{g\in
G^{tors}_{conj}}(\alpha\mapsto \eta^{QE}_g(G, V)(\alpha\cdot
v)),\label{GetaQEll}
\end{equation}
is well-defined and continuous. Moreover, $\eta^{QE}:
S\longrightarrow QE$ with $QE(G, V)$ defined in (\ref{EGVnog}) is
well-defined. \end{remark} Next, we construct the multiplication
map $\mu^{QE}$. First we define a map $$\mu^{QE}_{(g, h)}((G,V),
(H, W)): QE_g(G, V)\wedge QE_h(H, W)\longrightarrow QE_{(g,
h)}(G\times H, V\oplus W)$$ by sending a point $[t_1a_1+
t_2b_1]\wedge [u_1a_2 + u_2b_2]$ to
\begin{equation}
\begin{cases}[(1-\sqrt{t^2_2+u^2_2})\mu^F_{(g, h)}((G,V)
, (H, W))(a_1\wedge a_2) &\text{if $t^2_2+u^2_2\leq 1$ and
$t_2u_2\neq 0$;}\\ + \sqrt{t^2_2+u^2_2}(b_1+b_2)],
& \\
[(1-t_2)\mu^F_{(g, h)}((G,V), (H, W))(a_1\wedge a_2)+ t_2b_1],
 &\text{if $u_2=0$ and $0<t_2<1$}; \\ [(1-u_2)\mu^F_{(g, h)}((G,V), (H, W))(a_1\wedge a_2) + u_2b_2],
 &\text{if $t_2=0$ and $0<u_2<1$; } \\ [1\mu^F_{(g, h)}((G,V), (H, W))(a_1\wedge a_2) + 0],
 &\text{if $u_2=0$ and $t_2=0$; }\\ x_{g, h},
 &\text{Otherwise.}\end{cases}\label{muEg}
\end{equation} where $\mu^{QE}_{(g, h)}((G,V), (H, W))$ is the one
defined in (\ref{muFKg}) and $x_{g, h}$ is the basepoint of
$QE_{(g, h)}(G\times H, V\oplus W)$.

\begin{lemma}The map $\mu^{QE}_{(g, h)}((G,V), (H, W))$ defined in (\ref{muEg}) is well-defined and
continuous. \label{muEcontinuous}
\end{lemma}
The proof of Lemma \ref{muEcontinuous} is in Appendix
\ref{muEcontinuousProof}.

The basepoint of $QE(G, V)$ is the product of the basepoint of
each factor $\mbox{Map}_{C_G(g)}(G, QE_g(G, V))$, i.e. the product
of the constant map to the base point of  each $QE_g(G, V)$.

We can define the multiplication $\mu^{QE}((G, V), (H, W)): QE(G,
V)\wedge QE(H, W)\longrightarrow QE(G\times H, V\oplus W)$ by
$$\big(\!\!\prod\limits_{g\in G^{tors}_{conj}}
\!\!\alpha_g\big)\wedge\big(\!\!\prod\limits_{ h\in
H^{tors}_{conj}}\!\!\beta_h\big)\mapsto\prod\limits_{\substack{g\in
G^{tors}_{conj}\\ h\in H^{tors}_{conj}}}\!\!\bigg(\!\!(g',
h')\mapsto \mu^{QE}_{(g, h)}((G, V), (H,
W))\big(\alpha_g(g')\wedge\beta_h(h')\big)\!\!\bigg).
$$
\begin{lemma}
Let $G$, $H$, $K$ be compact Lie groups. Let $V$ be an orthogonal
$G-$representation, $W$ an orthogonal $H-$representation, and $U$
an orthogonal $K-$representation. Let $g\in G^{tors}$, $h\in
H^{tors}$, and $k\in K^{tors}$. Then we have the commutative
diagrams below.
\begin{equation}\begin{CD}
S^V\wedge S^W     @>\eta^{QE}_g(G, V)\wedge \eta^{QE}_h(H, W)>>  QE_g(G, V)\wedge QE_h(H, W)\\
@VV\cong V        @VV{\mu^{QE}_{(g, h)}((G, V), (H, W))}V\\
S^{V\oplus W}     @>\eta^{QE}_{(g, h)}(G\times H, V\oplus W)>>
QE_{(g, h)}(G\times H, V\oplus W)
\end{CD}\end{equation}
\begin{equation}
\begin{CD}
QE_g(G, V)\wedge QE_h(H, W)\wedge QE_k(K, U)     @>\mu^{QE}_g((G, V), (H, W))\wedge Id>>  QE_{(g, h)}(G\times H, V\oplus W)\wedge QE_k(K, U)\\
@VV Id\wedge\mu^{QE}_{(h, k)}(H\times K, W\oplus U)V        @V{\mu^{QE}_{((g, h), k)}((G\times H, V\oplus W), (K, U))}VV\\
QE_g(G, V)\wedge QE_{(h, k)}(H\times K, W\oplus U) @>\mu^{QE}_{(g,
(h, k))}((G, V), (H\times K, W\oplus U))>> QE_{(g, h, k)}(G\times
H\times K, V\oplus W\oplus U)
\end{CD}
\end{equation}
\begin{equation}
\begin{CD}
S^V\wedge QE_{h}(H, W)     @>{\eta^{QE}_g(G, V)\wedge Id}>>  QE_g(G, V)\wedge QE_h(H, W)  @>{\mu^{QE}_{(g, h)}((G,V),  (H, W))}>> QE_{(g, h)}(G\times H, V\oplus W)\\
@VV\tau V     @.   @VV {QE_{(g, h)}(\tau)}V\\
QE_h(H, W)\wedge S^V     @>Id\wedge\eta^{QE}_g(G, V)>>    QE_h(H,
W)\wedge QE_g(G, V) @>{\mu^{QE}_{(h, g)}((H, W), (G, V))}>>
QE_{(h, g)}(H\times G, W\oplus V)
\end{CD}
\end{equation}
Moreover, we have \begin{equation}\mu^{QE}_{(g, h)}((G, V), (H,
W))(x\wedge y)= \mu^{QE}_{(h, g)}((H, W), (G, V))(y\wedge x)
\label{quasicommEg}\end{equation} for any $x\in QE_g(G, V)$ and
$y\in QE_h(H, W)$. \label{Gorthocommdiaglemma}
\end{lemma}

The proof of Lemma \ref{Gorthocommdiaglemma} is straightforward
and is in Appendix \ref{GorthocommdiaglemmaProof}.

\begin{theorem}
Let $\Delta_G: G\longrightarrow G\times G$ be the diagonal map
$g\mapsto (g, g)$. For $G-$representations $V$ and $W$, let
$(\Delta_G)^*_{V\oplus W}: QE(G\times G, V\oplus W)\longrightarrow
QE(G, V\oplus W)$ denote the restriction map defined by the
formula (\ref{restrictionER}). Then
$QE:\mathcal{I}_G\longrightarrow G\mathcal{T}$ together with the
unit map $\eta^{QE}$ defined in (\ref{GetaQEll}) and  the
multiplication $\Delta_G^*\circ\mu^{QE}((G, -), (G, -))$ gives a
commutative $\mathcal{I}_G-$FSP that weakly represents
$QE^*_G(-)$. \label{GorthoQEll}\end{theorem}
\begin{proof}Let $G$,  $H$,   $K$ be compact Lie groups, $V$ an orthogonal
$G-$representation, $W$ an orthogonal $H-$representation and $U$
an orthogonal $K-$representation.

Let
$$X=\prod\limits_{g\in G^{tors}_{conj}} \alpha_g\in QE(G, V)\mbox{;
      }Y=\prod\limits_{h\in H^{tors}_{conj}} \beta_h\in QE(H, W)
      \mbox{;
      }Z=\prod\limits_{k\in K^{tors}_{conj}} \gamma_k\in QE(K, U).
$$

First we check the diagram of unity commutes. Let $v\in S^V$ and
$w\in S^W$.
$\mu^{QE}((G, V), (H, W))\circ (\eta^{QE}(G, V)\wedge \eta^{QE}(H,
W))(v\wedge w)$ is 
\begin{equation}\prod\limits_{g\in G^{tors}_{conj}, h\in H^{tors}_{conj}}\!\!\bigg(\!\! (g',
h')\mapsto \mu^{QE}_{(g, h)}((G, V), (H, W))\circ
(\eta^{QE}_{g}(G, V)\wedge \eta_{h}^{QE}(H, W))(g'\cdot v\wedge
h'\cdot w)\!\!\bigg)
.\label{TGorthocommdiaglemunity1}\end{equation} $\eta^{QE}(G\times
H, V\oplus W)(v\wedge w)= \prod\limits_{\substack{g\in
G^{tors}_{conj} \\ h\in H^{tors}_{conj}}}\bigg( (g', h')\mapsto
\eta^{QE}_{(g, h)}(G\times H, V\oplus W)(g'\cdot v\wedge h'\cdot
w)\bigg) ,$  is equal to  (\ref{TGorthocommdiaglemunity1}) by
Lemma \ref{Gorthocommdiaglemma}.

Next we check the diagram of associativity commutes.

$\mu^{QE}((G\times H, V\oplus W), (K, U))\circ (\mu^{QE}((G, V),
(H, W))\wedge Id)(X\wedge Y\wedge Z)$ is
\begin{align*}&\prod\limits_{g\in G^{tors}_{conj} h\in H^{tors}_{conj}, k\in G^{tors}_{conj}}\bigg((g', h', k')\mapsto\\ &\mu^{QE}_{((g, h), k)}((G\times H, V\oplus W), (K, U))\circ
(\mu^{QE}_{(g, h)}((G, V), (H, W))\wedge Id)(\alpha_g(g')\wedge
\beta_h(h')\wedge \gamma_k(k'))\bigg)\end{align*} And
$\mu^{QE}((G, V), (H\times K, W\oplus
U))\circ(Id\wedge\mu^{QE}(H\times K, W\oplus U))(X\wedge Y\wedge
Z)$ is \begin{align*}&\prod\limits_{g\in G^{tors}_{conj} h\in
H^{tors}_{conj}, k\in G^{tors}_{conj}}\bigg((g', h', k')\mapsto\\
&\mu^{QE}_{(g, (h, k))}((G, V), (H\times K, W\oplus
U))\circ(Id\wedge\mu^{QE}_{(h, k)}(H\times K, W\oplus
U))(\alpha_g(g')\wedge \beta_h(h')\wedge
\gamma_k(k'))\bigg)\end{align*}

By Lemma \ref{Gorthocommdiaglemma}, the two terms are equal.

In addition, $QE(\tau)\circ \mu^{QE}((G,V), (H, W))\circ
(\eta^{QE}(G, V)\wedge Id)(v\wedge X)$ is
$$\prod\limits_{g\in G^{tors}_{conj}, h\in H^{tors}_{conj}}\bigg(
(h', g')\mapsto QE_{(g, h)}(\tau)\circ \mu^{QE}_{(g, h)}((G,V),
(H, W))\circ (\eta^{QE}_g(G, V)\wedge Id)((g'\cdot v)\wedge
\beta_h(h'))\bigg)$$ And $\mu^{QE}((H, W), (G,
V))\circ(Id\wedge\eta^{QE}(H, W)) \circ \tau (v\wedge X)$ is
$$\prod\limits_{g\in G^{tors}_{conj}, h\in H^{tors}_{conj}}\bigg(
(h', g')\mapsto \mu^{QE}_{(h, g)}((H, W), (G,
V))\circ(Id\wedge\eta^{QE}_h(H, W)) \circ \tau ((g'\cdot v)\wedge
\beta_h(h'))\bigg)$$

The two terms are equal.  So the centrality of unit diagram
commutes.

Moreover, by Lemma \ref{Gorthocommdiaglemma}, $\mu^{QE}((G, V),
(H, W))(X\wedge Y)= $
\begin{align*}&\prod\limits_{g\in
G^{tors}_{conj}, h\in H^{tors}_{conj}}\bigg( (g',
h')\mapsto\mu^{QE}_{(g, h)}((G, V), (H, W))(\alpha_g(g')\wedge
\beta_h(h'))\bigg)\\ = &\prod\limits_{g\in G^{tors}_{conj}, h\in
H^{tors}_{conj}}\bigg( (h',g')\mapsto\mu^{QE}_{(h, g)}((H, W), (G,
V))(\beta_h(h')\wedge\alpha_g(g'))\bigg)\end{align*} is
$\mu^{QE}((H, W), (G, V))(Y\wedge X)$.  Therefore we have the
commutativity of $QE$.  \end{proof}

\begin{proposition}
Let $G$ be any compact Lie group. Let $V$ be an ample orthogonal
$G-$representation and $W$ an orthogonal $G-$representation. Let
$\sigma^{QE}_{G, V, W}: S^W\wedge QE(G, V)\longrightarrow QE(G,
V\oplus W)$ denote the structure map of $QE$ defined by the unit
map  $\eta^{QE}(G, V)$. Let $\widetilde{\sigma}^{QE}_{G, V, W}$
denote the right adjoint of $\sigma^{QE}_{G, V, W}$. Then
$\widetilde{\sigma}^{QE}_{G, V, W}: QE(G, V)\longrightarrow
\mbox{Map}(S^W, QE(G, V\oplus W))$ is a $G-$weak equivalence.
\label{fibrantweakequiv}\end{proposition}
\begin{proof}From the formula  of $\eta^{QE}(G, V)$, we can get an explicit
formula for $$\widetilde{\sigma}^{QE}_{G, V, W}:  QE(G,
V)\longrightarrow \mbox{Map}(S^W, QE(G, V\oplus W)).$$ Let
$\alpha:= \prod\limits_{g\in G^{tors}_{conj}} \alpha_g$ be any
element in $QE(G, V)=\prod\limits_{g\in G^{tors}_{conj}}
\mbox{Map}_{C_G(g)}(G, QE_g(G, V)),$  and $w$ an element in $S^W$.
For each $g\in G^{tors}_{conj}$, $w$ has a unique decomposition
$w=w_g^1\wedge w_g^2$ with $w_g^1\in S^{W^g}$ and $w_g^2\in
S^{(W^g)^{\perp}}$. $\widetilde{\sigma}^{QE}_{G, V, W}$ sends
$\alpha$ to
$$w\mapsto \bigg(\prod\limits_{g\in G^{tors}_{conj}}g'\mapsto \Delta_G^*\circ\mu^{QE}_{(g, g)}((G, V), (G, W))(\alpha_g(g'), \eta^{QE}_{g}(G, W) (g'\cdot w))\bigg).$$

It suffices to show that for each $g\in G^{tors}_{conj}$, the map
\begin{align*}\widetilde{\sigma}^{QE}_{G, g, V, W}: QE_g(G, V)&\longrightarrow \mbox{Map}_{C_G(g)}(S^W, QE_g(G,
V\oplus W))\\ x &\mapsto \bigg(w \mapsto
\Delta_G^*\circ\mu^{QE}_{(g, g)}((G, V), (G, W))(x, \eta^{QE}_g(G,
W)(w))\bigg)\end{align*} is a $C_G(g)-$weak equivalence. We check
for each closed subgroup $H$ of $C_G(g)$, the map
$(\widetilde{\sigma}^{QE}_{G, g, V, W})^H$ on the fixed point
space is a homotopy equivalence.

\textbf{Case I:}  $g\in H$.

$QE_g(G, V)^H$ is the space $F_g(G, V)^H$. By Proposition
\ref{newF}, $$\widetilde{\sigma}_g(G, V, W)^H: F_g(G,
V)^H\longrightarrow \mbox{Map}_H(S^{W^g}, F_g(G, V\oplus W))$$ is
a weak equivalence.



By Theorem \ref{KUM}, $$\mbox{Map}_H(S^{W}, QE_g(G, V\oplus
W))\longrightarrow \mbox{Map}_H(S^{W^g}, F_g(G, V\oplus W)),
\mbox{   }  f\mapsto f|_{S^{W^g}}$$ is a homotopy equivalence. And
we have the diagram below commutes.
\begin{equation}\xymatrix{F_g(G,
V)^H\ar[r]^>>>>>{\simeq}\ar[rd] &\mbox{Map}_H(S^{W^g}, F_g(G, V\oplus W))\\
&\mbox{Map}_H(S^{W}, QE_g(G, V\oplus
W))\ar[u]_{\simeq}}\end{equation} So $\widetilde{\sigma}^{QE}_{G,
g, V, W}\!\!:\!\! F_g(G, V)^H\!\!\longrightarrow \!\!
\mbox{Map}_H(S^{W^g}, F_g(G, V\oplus W))$ is a homotopy
equivalence. 

\textbf{Case II:   }  $g$ is not in $H$.
In this case, $QE_g(G, V)^H$ is contractible. It suffices to show
 $\mbox{Map}_H(S^{W}, QE_g(G, V\oplus W)$ is also
contractible. Note that for any closed subgroup $H'$ of $H$,
$QE_g(G, V\oplus W)^{H'}$ is contractible. So for each $n-$cell
$H/H'\times D^n$ of $S^W$, it's mapped to $QE_g(G, V\oplus
W)^{H'}$ unique up to homotopy. So $\mbox{Map}_H(S^{W}, QE_g(G,
V\oplus W)$ is contractible.

Therefore $\widetilde{\sigma}^{QE}_{G, g, V, W}$ is a
$C_G(g)-$weak equivalence. So $\widetilde{\sigma}^{QE}_{G, V, W}$
is a $G-$weak equivalence.
\end{proof}

By Proposition \ref{QEinal} and Proposition
\ref{fibrantweakequiv} we can get the conclusion below.
\begin{corollary}
For any compact Lie group $G$,  $(QE(G, -), \eta^{QE}, \mu^{QE})$  represents
$QE^{(-)}_G(-)$ in $GwS$. \label{uptokweak}
\end{corollary}
Especially we have the conclusion for quasi-elliptic cohomology.
\begin{corollary}
For any compact Lie group $G$,  $(QKU(G, -), \eta^{QKU}, \mu^{QKU})$  represents
$QEll^{(-)}_G(-)$ in $GwS$.
\end{corollary}
At last, we get the main conclusion of Section
\ref{GorthospectraQEll}.
\begin{theorem}
There is a well-defined functor $\mathcal{Q}$ from the full subcategory consisting of
$\mathcal{I}_G-$FSP in $GwS$ to the same category sending $(E_G, \eta^E, \mu^E)$  to
$(QE(G, -), \eta^{QE}, \mu^{QE})$ that represents the
cohomology theory $QE$ in $GwS$.

The restriction of $\mathcal{Q}$ to the full subcategory consisting of
commutative $\mathcal{I}_G-$FSP  is a functor from that category  to itself.
 \end{theorem}

\section{The Restriction map}\label{gaction} In this section we
construct the restriction maps $QE(G, V)\longrightarrow QE(H, V)$
for group homomorphisms $H\longrightarrow G$. The restriction maps
for quasi-elliptic cohomology can be constructed in the same way.


Let $\phi: H\longrightarrow G$ be a group homomorphism and $V$ a
$G-$representation. For any homomorphism of compact Lie groups
$\phi: H\longrightarrow G$ and $H-$space $X$, we have the
change-of-group isomorphism $QE^*_G(G\times_H X)\cong QE^*_H(X).$
Thus, for any subgroup $K$ of $H$, we have the isomorphism
$QE^n_G(G/K)=QE^n_G(G\times_H H/K)\cong QE^n_H(H/K).$ So by
Proposition \ref{QEinal} the space $QE(G, V)^K$ is homotopy
equivalent to $QE(H, V)^K$ when $V$ is a faithful
$G-$representation. It implies when we consider $QE(G, V)$ as an
$H-$space, it is $H-$weak equivalent to $QE(H, V)$.

As indicated in Remark \ref{Enotorthosp}, the orthogonal
$G-$spectrum $QE(G, -)$  cannot arise from an orthogonal spectrum.
As a result, the restriction map $QE(G, V)\longrightarrow QE(H, V)$ 
cannot be a homeomorphism. We construct in this section a
restriction map $\phi^*_V$ that is $H-$weak equivalence such that
the diagram below commutes.\begin{equation}\begin{CD}\pi_k(QE(G, V)) @>{\cong}>> QE^V_G(S^k)\\
@VV{\pi^k(\phi^*_V)}V     @VV{\phi^*}V \\
\pi_k(QE(H, V)) @>{\cong}>>
QE^V_H(S^k)\end{CD}\label{restrictionupto}
\end{equation}
where $\phi^*$ is the restriction map of quasi-elliptic
cohomology.

Let $X$ be a $G-$space. Let $g\in G^{tors}$ and $h\in H^{tors}$.
The group homomorphism $\phi: H\longrightarrow G$ sends $C_H(h)$
to $C_{G}(g)$ and also gives
$$\phi_*: \Lambda_H(h)\longrightarrow \Lambda_{G}(\phi(h)), \mbox{    } [h', t]\mapsto [\phi(h'), t].$$
$\phi$ induces an $H-$action on $X$. Especially, $X^g=X^h$ and
$\phi_*$ induces a $\Lambda_H(h)-$action on it for each $h\in
H^{tors}$. We consider the equivalent definition of the
$QE-$theory
$$QE^*_G(X)=\prod_{g\in G^{tors}} E^*_{\Lambda_G(g)}(X^g).$$ With this
definition, the restriction map can have a relatively simple form.

For each $g\in G^{tors}$, we first define a map
$$Res_{\phi, g}: \mbox{Map}_{C_{G}(g)}(G, QE_g(G, V))\longrightarrow
\prod\limits_{\tau}\mbox{Map}_{C_{H}(\tau)}(H, QE_{\tau}(H, V))$$
in the form $\prod\limits_{\tau}\bigg( R_{\phi, \tau}:
\mbox{Map}_{C_{G}(g)}(G, QE_g(G, V))\longrightarrow
\mbox{Map}_{C_{H}(\tau)}(H, QE_{\tau}(H, V))\bigg)$ where $\tau$
goes over all the elements $\tau$ in $H^{tors}$
such that $\phi(\tau)=g$. 
Then we will combine all the $Res_{\phi, g}$s to define the
restriction map $\phi^*_V$.

The restriction map
$$\phi^*_V: QE(G, V)\longrightarrow QE(H, V)$$ to be defined should make the diagram
(\ref{resmoral}) commute, which implies that (\ref{restrictionupto}) commutes. 
\begin{equation}\xymatrix{X^g\ar[d]_{=} \ar[r]&X\ar[d]_{=}\ar[r]^>>>>{\widetilde{f}}
&\mbox{Map}_{C_G(g)}(G, QE_g(G, V))\ar[d]^{R_{\phi,
\tau}}\ar[r]^>>>>>{\alpha\mapsto
\alpha(e)}&F_g(G, V)\ar[d]^{res|^{\Lambda_G(g)}_{\Lambda_H(\tau)}}\\
X^{\tau}\ar[r]&X\ar[r]^>>>>{R_{\phi, \tau}\circ\widetilde{f}}
&\mbox{Map}_{C_H(\tau)}(H, QE_\tau(H, V))\ar[r]^>>>>>{\beta\mapsto
\beta(e)}&F_\tau(H, V)}\label{resmoral}\end{equation} where
$res|^{\Lambda_G(g)}_{\Lambda_H(\tau)}$ is the restriction map
defined in (\ref{EGres}).



Let $\tau\in H^{tors}$ and $g=\phi(h)$. Then we have the
isomorphism $$a_{\tau}: (V)^{\mathbb{R}}_g\oplus
V^g\longrightarrow (V)^{\mathbb{R}}_{\tau}\oplus V^{\tau}$$
sending $v$ to $v$. For any $[b, t]\in \Lambda_H(h)$,
$a_{\tau}([\phi(b), t]v)=[b, t]a_{\tau}(v).$

In addition, we have the restriction map
$res|^{\Lambda_G(g)}_{\Lambda_H({\tau})}: F_g(G, V)\longrightarrow
F_{\tau}(H, V)$ defined as below. Let $\beta:
S^{(V)^{\mathbb{R}}_g}\longrightarrow E((V)^{\mathbb{R}}_g\oplus
V^g)$ be an $\mathbb{R}-$equivariant map. Note that
$S^{(V)^{\mathbb{R}}_{\tau}}$ and $S^{(V)^{\mathbb{R}}_g}$ have
the same underlying space, and $(V)^{\mathbb{R}}_g\oplus V^g$ and
$(V)^{\mathbb{R}}_{\tau}\oplus V^{\tau}$ have the same underlying
vector space. $res|^{\Lambda_G(g)}_{\Lambda_H({\tau})}(\beta)$ is
defined to be the composition
\begin{equation}\begin{CD}S^{(V)^{\mathbb{R}}_{\tau}} @>{x\mapsto x}>> S^{(V)^{\mathbb{R}}_g} @>{\beta}>> E((V)^{\mathbb{R}}_g\oplus V^g)
@>{E(a_{\tau})}>> E((V)^{\mathbb{R}}_{\tau}\oplus
V^{\tau})\end{CD}\label{EGres}\end{equation}
which is the identity map on the underlying spaces.

Let $\psi: K\longrightarrow H$ be another group homomorphism and
$\psi(k)=h$ for some $k\in K$. Then we have
\begin{equation}res|^{\Lambda_H(h)}_{\Lambda_K(k)}\circ res|^{\Lambda_G(g)}_{\Lambda_H(h)}
= res|^{\Lambda_G(g)}_{\Lambda_K(k)}\label{FRassos}\end{equation}

Note $S(G, V)_g$ has the same underlying space as $S(H, V)_\tau$.
Consider the join of maps
\begin{equation}res|^{\Lambda_G(g)}_{\Lambda_H(\tau)}\ast Id: F_g(G, V)\ast  S(G, V)_g\longrightarrow
F_{\tau}(H, V)\ast S(H, V)_{\tau}\end{equation} It is the identity
map on the underlying space and has the equivariant property: for
any $a\in C_H(\tau)$, $x\in H$,
\begin{equation}res|^{\Lambda_G(g)}_{\Lambda_H(\tau)}\ast
Id(\phi(a)\cdot x)=a\cdot
res|^{\Lambda_G(g)}_{\Lambda_H(\tau)}\ast
Id(x).\label{abtauequiv}\end{equation}

$res|^{\Lambda_G(g)}_{\Lambda_H(\tau)}\ast b_{\tau}$ gives a
well-defined map on the quotient space $r_{\phi, \tau}: QE_g(G,
V)\longrightarrow QE_{\tau}(H, V)$. It also has the equivariant
property as (\ref{abtauequiv}).
For any $\rho$ in $\mbox{Map}_{C_{G}(g)}(G, E_g(G, V))$, let
$R_{\phi, \tau}(\rho)$ be the composition $\begin{CD}H @>{\phi}>>
G @>{\rho}>>  QE_g(G, V) @>{r_{\phi, \tau}}>> QE_{\tau}(H, V).
\end{CD}$
$R_{\phi, \tau}(\rho)$ is $C_H(\tau)-$equivariant: $R_{\phi,
\tau}(\rho)(ah)=r_{\phi, \tau}(\rho( \phi(ah)))=r_{\phi,
\tau}(\rho( \phi(a)\phi(h)))=ar_{\phi, \tau}(\rho(\phi(h)))
=a\cdot R_{\phi, \tau}(\rho)(h),$ for any $a\in C_H(\tau)$, $h\in
H$.

For any  $g\in Im\phi$, $Res_{\phi, g}$  is defined to be
$\prod\limits_{\tau}R_{\phi, \tau}$ where $\tau$ goes over all the
$\tau\in H^{tors}$ such that $\phi(\tau)=g$.  The restriction map
is defined to be
\begin{equation}\phi^*_V:=\prod\limits_{g} Res_g: QE(G, V)\longrightarrow
QE(H, V)\label{restrictionER}\end{equation} where $g$ goes over
all
the elements in $G^{tors}$ in the image of  $\phi$. 

\begin{lemma}
(i) $R_{\phi, \tau}$ is the restriction map making the diagram
\begin{equation}\xymatrix{\mbox{Map}_{C_G(g)}(G,
QE_g(G,V))\ar[d]_{R_{\phi, \tau}}\ar[r]^>>>>>{\alpha\mapsto
\alpha(e)}&F_g(G, V)\ar[d]^{res|^{\Lambda_G(g)}_{\Lambda_H(\tau)}}\\
\mbox{Map}_{C_H(\tau)}(H, QE_{\tau}(H,
V))\ar[r]^>>>>>{\beta\mapsto \beta(e)}&F_{\tau}(H,
V)}\label{babyres}\end{equation} commute. So the restriction map
$\phi^*_V$ makes the diagram (\ref{resmoral}) commute.

(ii)Let $\phi: H\longrightarrow G$ and $\psi: K\longrightarrow H$
be two group homomorphism and $V$ a $G-$representation. Then
$\psi_V^*\circ\phi_V^*=(\phi\circ\psi)_V^*.$ The composition is
associative.

(iii) $Id_V^*: QE(G, V)\longrightarrow QE(G, V)$ is the identity
map. \label{restrictionproperties}
\end{lemma}
\begin{proof}
(i) $R_{\phi, \tau}(\alpha)(e)=r_{\phi,
\tau}\circ\alpha(e)=res|^{\Lambda_G(g)}_{\Lambda_H(\tau)}\alpha(e).$
So (\ref{babyres}) commutes.

(ii) Let $\rho_g: G\longrightarrow QE_g(G, V)$ be a
$C_G(g)-$equivariant map for each $g\in G^{tors}$. Note that if we
have $\psi(\sigma)=\tau$ and $\phi(\tau)=g$, then $r_{\phi,
\tau}\circ r_{\psi, \sigma}= r_{\phi\circ\psi, \sigma}$ since both
sides are identity maps on the underlying spaces. Then we have
for any $k\in K$,
\begin{align*}\psi_V^*\circ\phi_V^*(\prod\limits_{g\in
G^{tors}}\rho_g)&=\prod\limits_{g}\prod\limits_{\tau}\prod\limits_{\sigma}r_{\psi,
\sigma}\circ r_{\phi, \tau}\rho_g(\phi(\psi(k))) \\ &=
\prod\limits_{g}\prod\limits_{\tau}\prod\limits_{\sigma}
r_{\psi\circ\phi, \sigma}\rho_g(\phi\circ\psi(k))=
(\phi\circ\psi)_V^*(\prod\limits_{g\in
G^{tors}}\rho_g)\end{align*}  where  $\tau$ goes over all the
elements in $H^{tors}$ with $\phi(\tau)=g$ and $\sigma$ goes over
all the elements in $K^{tors}$ with $\psi(\sigma)=\tau$. So
$\psi_V^*\circ\phi_V^*=(\phi\circ\psi)_V^*.$

(iii) For the identity map $Id: G\longrightarrow G$, by the
formula of the restriction map,  $Id_V^*(\prod\limits_{g\in
G^{tors}}\rho_g)= \prod\limits_{g\in G^{tors}}\rho_g$, thus, is
the identity.
\end{proof}

\section{The Birth of a new global homotopy theory}\label{newglobal}

At the early beginning of equivariant homotopy theory people noticed that certain theories naturally exist not only for one particular group but for all groups
in a specific class. This observation motivated the birth of global homotopy theory. In \cite{SS} the concept of orthogonal spectra is introduced, which is defined from $\mathbb{L}-$functors with $\mathbb{L}$ the category of inner product real spaces. Each global spectra consists of compatible $G-$spectra with $G$ across the entire category of groups and they reflect any symmetry.
Globalness is a measure of the naturalness of a cohomology theory.

In Remark 4.1.2 \cite{SS}, Schwede discussed the relation between orthogonal $G-$spectra and global spectra. We have
the question whether the underlying orthogonal $G-$spectrum of the $I_G-$FSP $(QE(G, -), \eta^{QE}, \mu^{QE})$ in Theorem \ref{GorthospectraQEll}
can arise from an orthogonal spectrum.
 Ganter showed that
$\{QEll^*_G\}_G$ have the change-of-group
isomorphism, which is a good sign that quasi-elliptic cohomology may be globalized.

By the discussion in Remark 4.1.2 \cite{SS},
however, the answer to this question is no.  A
$G-$spectrum $Y$ is isomorphic to an orthogonal $G-$spectrum of
the form
 $X\langle G\rangle$ for some orthogonal spectrum $X$ if and only if for every trivial $G-$representation $V$ the $G-$action on
$Y(V)$ is trivial. $QE(V)$ is not trivial when $V$ is trivial. So
it cannot arise from an orthogonal spectrum.

Then it is even more difficult to see whether each elliptic cohomology theory,
whose form is more intricate and mysterious than quasi-elliptic cohomology, can be globalized in the current setting.

Our solution is to establish a more flexible global homotopy theory where
quasi-elliptic chomomology can fit into. We hope that it is easier to judge whether a cohomology theory, especially an
elliptic cohomology theory, can be globalized in the new theory. In addition we want to show that the new global homotopy theory is equivalent to the current global homotopy theory.


We construct in \cite{Huanthesis} a category $D_0$ to replace $\mathbb{L}$ whose objects
are $(G, V,\rho)$ with  $V$ an inner product vector space, $G$ a
compact group and $\rho$ a faithful group representations
$$\rho: G\longrightarrow O(V),$$ and whose morphism  $\phi=(\phi_1, \phi_2): (G,
V, \rho)\longrightarrow (H, W, \tau)$ consists of a linear
isometric embedding $\phi_2: V\longrightarrow W$ and a group
homomorphism $\phi_1: \tau^{-1}(O(\phi_2(V)))\longrightarrow G$,
which makes the diagram (\ref{Dmor}) commute.
\begin{equation}\xymatrix{G\ar[r]^{\rho}
&O(V)\ar[d]^{\phi_{2*}}\\
\tau^{-1}(O(\phi_2(V)))\ar[u]^{\phi_1}\ar[r]^>>>>>{\tau}
&O(W)}\label{Dmor}\end{equation} In other words, the group action
of $H$ on $\phi_2(V)$ is induced from that of $G$. Intuitively,
the category $D_0$ is obtained by adding the restriction maps
between representations into the category $\mathbb{L}$.

Instead of the category of orthogonal spaces, we study the category of $D_0-$spaces. 
The category
of orthogonal spaces is a full subcategory of the category $D_0T$ of
$D_0-$spaces.
Apply the idea of diagram spectra in \cite{MMSS}, we can also
define $D_0-$spectra and $D_0-$FSP.

Combining the orthogonal $G-$spectra of quasi-elliptic cohomology
together, we get a well-defined $D_0-$spectra and $D_0-$FSP. Thus,
we can define global quasi-elliptic cohomology in the category of
$D_0-$spectra.

\begin{theorem}(Theorem 7.2.3, \cite{Huanthesis})There is a $D_0-$FSP weakly representing quasi-elliptic cohomology. \end{theorem}

Equipping a homotopy theory with a model structure is like interpreting the world via philosophy. Model category theory is
an essential basis and tool to judge whether two homotopy theories describe the same world. We build several model structures on $D_0T$. First by the
theory in \cite{MMSS}, there is a level model structure on $D_0T$.
\begin{theorem} (Theorem 6.3.4, \cite{Huanthesis})

The category of $D_0-$spaces is a compactly generated topological
model category with respect to the level equivalences, level
fibrations and $q-$cofibrations. It is right proper and
left proper.

\end{theorem}

$D_0$ is a generalized Reedy category in the sense of \cite{BM}.
We can construct a Reedy model structure on $D_0T$.

\begin{theorem}(Theorem 6.4.5, \cite{Huanthesis})The Reedy cofibrations, Reedy weak equivalences and Reedy
fibrations form a model structure, the Reedy model structure, on
the category of $D_0-$spaces. \label{Reedy}
\end{theorem}

We are constructing a global model structure on $D_0T$ Quillen
equivalent to the global model structure on the orthogonal spaces
constructed by Schwede in \cite{SS}. Moreover, other than the new unstable global homotopy theory, we
will also establish the new stable global homotopy theory.

\appendix
\section{Join}\label{join}
\subsection{Definition}

\begin{definition}
In topology, the join $A\ast B$ of two topological spaces $A$ and
$B$ is defined to be the quotient space $(A \times B \times [0,
1])/ R,$ where $R$ is the equivalence relation generated by $(a,
b_1, 0) \sim (a, b_2, 0)\mbox{  for all } a\in A \mbox{ and } b_1,
b_2 \in B$ and $(a_1, b, 1) \sim (a_2, b, 1)  \mbox{  for all }
a_1, a_2 \in A \mbox{ and } b \in B.$

At the endpoints, this collapses $A\times B\times \{0\}$ to $A$
and $A\times B\times \{1\}$ to $B$.
\end{definition}
The join $A\ast B$ is the homotopy colimit of the diagram
$\!\!\xymatrix{&A &A\times B\ar[r]\ar[l] &B}.$

A nice way to write points of $A\ast B$ is as formal linear
combination $t_1a+t_2b$ with $0\leq t_1, t_2\leq 1$ and
$t_1+t_2=1$, subject to the rules $0a+1b=b$ and $1a+0b=a$. The
coordinates correspond exactly to the points in $A\ast B$. 
\begin{proposition}
Join is associative and commutative. Explicitly, $A\ast (B\ast C)$
is homeomorphic to $(A\ast B)\ast C$, and $A\ast B$ is homeomrphic
to $B\ast A$. \label{ja}\end{proposition}

\subsection{Group Action on the Join}

\begin{example}
Let $G$ be a compact Lie group. Let $A$, $B$ be $G-$spaces. Then
$A\ast B$ has a $G-$structure on it by \begin{equation}g\cdot
(t_1a+t_2b):=t_1(g\cdot a)+ t_2(g\cdot b),\mbox{  for any }g\in G,
a\in A, b\in B, \mbox{  and   }t_1, t_2\geq 0,
t_1+t_2=1.\label{gj1}\end{equation}

It's straightforward to check (\ref{gj1}) defines a continuous
group action.

\label{gpj1}\end{example}

\begin{example}
Let $G$ and $H$ be compact Lie groups. Let $A$ be a $G-$space and
$B$ a $H-$space. Then $A\ast B$ has a continuous $G\times
H-$structure on it by \begin{equation}(g, h)\cdot
(t_1a+t_2b):=t_1(g\cdot a)+ t_2(h\cdot b),\mbox{  for any }g\in G,
a\in A, b\in B, \mbox{  and }t_1, t_2\geq 0,
t_1+t_2=1.\label{gj2}\end{equation} \label{gpj2}\end{example}

\section{Equivariant Orthogonal spectra}\label{recalortho}

In Section \ref{recallGortho}, we recall the basics of equivariant
orthogonal
 spectra. There are many references for
this topic, such as \cite{BG}, \cite{MM} \cite{SS}, \cite{OP},
etc. In Section \ref{globalK} we recall the global K-theory, which
is a prominent example of global homotopy theory. Its properties
will be applied in the construction of the orthogonal $G-$spectrum
for quasi-elliptic cohomology.

\subsection{Orthogonal G-spectra}\label{recallGortho}

Let $G$ be a compact Lie group. Let $\mathcal{I}_G$ denote the
category whose objects are pairs $(\mathbb{R}^n, \rho)$ with
$\rho$ a homomorphism from $G$ to $O(n)$ giving $\mathbb{R}^n$ the
structure of a $G-$representation. Morphisms $(\mathbb{R}^m,
\mu)\longrightarrow (\mathbb{R}^n, \rho)$ are linear isometric
isomorphisms $\mathbb{R}^m\longrightarrow\mathbb{R}^n$.

Let $Top_G$ denote the category with objects based $G-$spaces and
morphisms continuous based maps.

\begin{definition}An $\mathcal{I}_G-$space is a $G-$continuous functor $X: \mathcal{I}_G\longrightarrow Top_G$.
Morphisms between $\mathcal{I}_G-$spaces are natural
$G-$transformations. \end{definition}

\begin{definition}An orthogonal $G-$spectrum is an $\mathcal{I}_G-$space $X$ together with a natural transformation of functors
$\mathcal{I}_G\times\mathcal{I}_G\longrightarrow Top_G$
$$X(-)\wedge S^{-}\longrightarrow X(-\oplus -)$$ satisfying
appropriate associativity and unitality diagrams. In other words,
an orthogonal $G-$spectrum is an $\mathcal{I}_G-$space with an
action of the sphere $\mathcal{I}_G-$space.
\end{definition}

\begin{definition}
For $\mathcal{I}_G-$spaces $X$ and $Y$, define the "external"
smash product $X\overline{\wedge}Y$ by
\begin{equation}
X\overline{\wedge} Y=\wedge\circ (X\times Y):
\mathcal{I}_G\times\mathcal{I}_G\longrightarrow Top_G;
\end{equation}
thus $(X\overline{\wedge} Y)(V, W)=X(V)\wedge Y(W).$

\end{definition}

We have an equivariant notion of a functor with smash product
(FSP).
\begin{definition}
An $\mathcal{I}_G-$FSP is an $\mathcal{I}_G-$space $X$ with a unit
$G-$map $\eta: S\longrightarrow X$ and a natural product $G-$map
$\mu: X\overline{\wedge} X\longrightarrow X\circ\bigoplus$ of
functors $\mathcal{I}_G\times\mathcal{I}_G\longrightarrow Top_G$
such that the evident unit, associativity and centrality of unit
diagram also commutes.
\end{definition}

\begin{lemma}
An $\mathcal{I}_G-$FSP has an underlying $\mathcal{I}_G-$spectrum
with structure $G-$map $$\sigma=\mu\circ(id
\overline{\wedge}\eta): X\overline{\wedge}S\longrightarrow
X\circ\oplus.$$
\end{lemma}

\subsection{Orthogonal spectra}\label{ght}

The global homotopy theory is established to better describe
certain theories naturally exists not only for a particular group,
but for all groups of certain type in a compatible way. Some
prominent examples of this are equivariant stable homotopy,
equivariant K-theory, and equivariant bordism.

The idea of global orthogonal spectra was first inspired in the
paper \cite{GP} by Greenlees and May where they introduce the
concept of global $\mathcal{I}_*-$functors with smash product. The
idea is developed by Mandell and May \cite{MM} and Bohmann
\cite{BG}. Schwede develops another modern approach of global
homotopy theory using a different categorical framework in
\cite{SS}, which is the main reference for Section \ref{ght}. For
definition of orthogonal spectra in detail, please refer
\cite{OP}, \cite{MMSS}, \cite{SS}.

First we recall  the definition of orthogonal spaces. Let
$\mathbb{L}$ denote the category whose objects are inner product
real spaces and whose morphism set between two objects $V$ and $W$
are the linear isometric embeddings $L(V, W)$.

\begin{definition}An orthogonal space is a continuous functor $Y: \mathbb{L}\longrightarrow \mathcal{T}$ to the category of topological spaces. A morphism of orthogonal
spaces is a natural transformation. We denote by spc the category
of orthogonal spaces.\label{orthospdef}\end{definition}

Orthogonal spectra is the stabilization of orthogonal spaces.

Let $\mathbb{O}$ denote the category whose objects are inner
product real spaces and the morphisms $O(V, W)$ between two
objects $V$ and $W$ is the Thom space of the total space $$\xi(V,
W):=\{(w, \phi)\in W\times L(V, W)| W\perp \phi(V)\}$$ of the
orthogonal complement vector bundle, whose structure map $\xi(V,
W)\longrightarrow L(V, W)$ is the projection to the second factor.

\begin{definition}An orthogonal spectrum is a based continuous functor from $\mathbb{O}$ to the category  of based compactly generated weak Hausdorff spaces.
A morphism is a natural transformation of functors. Let $Sp$
denote the category of orthogonal spectrum.
\end{definition}

\begin{definition}
Given an orthogonal spectrum $X$ and a compact Lie group $G$, the
collection of $G-$spaces $X(V)$, for $V$ a $G-$representation, and
the equivariant structure maps $\sigma_{V,W}$ form an orthogonal
$G-$spectrum. This orthogonal $G-$spectrum $$X\langle
G\rangle=\{X(V), \sigma_{V, W}\}$$ is called \textit{the
underlying orthogonal $G-$spectrum} of $X$.
\label{underlyingGspectr}
\end{definition}

\subsection{Global K-theory and its variations} \label{globalK}

A classical example of orthogonal spectra is global K-theory.
Quasi-elliptic cohomology can be expressed in terms of equivariant
K-theory. And this example is especially important for our
construction.

In \cite{JK} Joachim constructs $G$-equivariant K-theory as an
orthogonal $G$-spectrum $\mathbb{K}_G$ for any compact Lie group
$G$. In fact it is the only known $E_{\infty}-$version of
equivariant complex K-theory when $G$ is a compact Lie group.

For any real $G-$representation $V$, let $\mathbb{C}l_V$ be the
Clifford algebra of $V$ and $\mathcal{K}_V$ be the $G-C^*-$algebra
of compact operators on $L^2(V)$. Let $s:= C_0(\mathbb{R})$ be the
graded $G-C^*-$algebra of continuous functions on $\mathbb{R}$
vanishing at infinity with trivial $G-$action. Then the orthogonal
$G-$spectrum for equivariant K-theory defined by  Joachim is the
lax monoidal functor given by
$$\mathbb{K}_G(V)=Hom_{C^*}(s,
\mathbb{C}l_V\otimes\mathcal{K}_V)$$ of $\mathbb{Z}/2-$graded
$\ast-$homomorphisms from $s$ to
$\mathbb{C}l_V\otimes\mathcal{K}_V$.

 Bohmann showed in her paper
\cite{BG} that Joachim's model is "global", i.e.  $\mathbb{K}$ is
an orthogonal $\mathcal{G}-$spectrum. For more detail, please read
\cite{BG} for reference.

Schwede's construction of global K-theory $KR$ in \cite{SS} is a
unitary analog of the construction by Joachim. It is an
ultra-commutative ring spectrum whose $G-$homotopy type realizes
Real $G-$equivariant periodic K-theory. He also shows that the
spaces in the orthogonal spectrum $KR$ represent Real equivariant
K-theory.

For any complex inner product space $W$, let $\Lambda(W)$ be the
exterior algebra $W$ and $Sym(W)$ the symmetric algebra of it. The
tensor product $$\Lambda(W)\otimes Sym(W)$$ inherits a hermitian
inner product from $W$ and it's $\mathbb{Z}/2-$graded by even and
odd exterior powers. Let $\mathcal{H}_W$ denote the Hilbert space
completion of $\Lambda(W)\otimes Sym(W)$. Let $\mathcal{K}_W$ be
the $C^*-$algebra of compact operators on $\mathcal{H}_W$. The
orthogonal spectrum $KR$ is defined to be the lax monoidal functor
$$KR(W)=Hom_{C^*}(s,\mathcal{K}_W).$$

Let $uW$ denote the underlying euclidean vector space of $W$.
There is an isomorphism of $\mathbb{Z}/2-$graded $C^*-$algebras
$$Cl(uW)\otimes_{\mathbb{R}}\mathcal{K}(L^2(W))\cong
\mathcal{K}_W.$$ So we get a homeomorphism $$KR(W)\cong
Hom_{C^*}(s,
Cl(uW)\otimes_{\mathbb{R}}\mathcal{K}(L^2(W)))=\mathbb{K}(uW).$$

We have the relations below between the global Real K-theory $KR$,
periodic unitary K-theory $KU$ and periodic orthogonal real
K-theory $KO$.
$$KU=u(KR);\mbox{  } KO=KR^{\psi}.$$ 
In \cite{SS},  Schwede shows that the spaces in the orthogonal
spectrum $KR$ represent real equivariant K-theory.

\begin{theorem}\label{globalKtheorem}
For a compact Lie group $G$, a "sufficiently large" (i.e.
faithful) real $G-$representation $V$ and a compact $G-$space $B$,
there is a bijection $\Psi_{G, B, V}: K_G(B)\longrightarrow [B_+,
KU(V)]^G$ that is natural in $B$.
\end{theorem}

We will use the orthogonal spectrum $KU$ in the construction of
orthogonal quasi-elliptic cohomology.
\begin{definition} An orthogonal $G-$representation is called ample if its
complexified symmetric algebra is complete complex $G-$universe.
\end{definition}

\begin{theorem}
(i) Let $G$ be a compact Lie group and $V$ an orthogonal
$G-$representation. For  every ample $G-$representation $W$, the
adjoint structure map
$$\widetilde{\sigma}^K_{V, W}: KU(V)\longrightarrow \mbox{Map}(S^W, KU(V\oplus
W))$$ is a $G-$weak equivalence.\\ (ii) Let $G$ be an augmented
Lie group and $V$ a real $G-$representation such that $Sym(V)$ is
a complete real $G-$universe. For every real $G-$representation
$W$ the adjoint structure map $$\widetilde{\sigma}^K_{V, W}:
KR(V)\longrightarrow \mbox{Map}(S^W, KR(V\oplus W))$$ is a
$G-$weak equivalence. \label{KUGweak}
\end{theorem}

\section{Faithful representation of $\Lambda_G(g)$}\label{reallambda}
We will apply the orthogonal spectrum $KU$ of global K-theory to
construct the orthogonal $G-$spectrum of $QE^*_G$. As indicated in
Theorem \ref{globalKtheorem}, we will need a faithful
$\Lambda_G(g)-$representation. Thus, before construction in
Section \ref{anotherweaksp} and \ref{Gorthospstructure}, we
discuss complex and real $\Lambda_G(\sigma)-$representations in
Section \ref{prere} and \ref{realdef} respectively.

\subsection{Preliminaries: faithful representations of
$\Lambda_{G}(g)$}\label{prere} As shown in Theorem
\ref{globalKtheorem}, $KU(V)$ represents $G-$equivariant complex
K-theory when $V$ is a faithful $G-$representation. In this
section, we construct a faithful
$\Lambda_G(\sigma)-$representation from a faithful
$G-$representation.

Let $G$ be a compact Lie group and  $\sigma\in G^{tors}$ with
order $l$. Let $\rho$ be a complex $G-$representation with
underlying space $V$. Let $i: C_G(\sigma)\hookrightarrow G$ denote
the inclusion. 
Let $\{\lambda\}$ denote all the irreducible complex
representations of $C_G(\sigma)$. As said in \cite{FH}, we have
the decomposition of a representation into its isotypic components
$i^*V\cong\bigoplus_{\lambda}V_{\lambda}$ where $V_{\lambda}$
denotes the sum of all subspaces of $V$ isomorphic to $\lambda$.
Each $V_{\lambda}=Hom_{C_G(\sigma)}(\lambda,
V)\otimes_{\mathbb{C}}\lambda$ is unique as a subspace. Note that
$\sigma$ acts on each $V_{\lambda}$ as a diagonal matrix.

Each $V_{\lambda}$ can be equipped with a
$\Lambda_G(\sigma)-$action. Each $\lambda(\sigma)$ is of the form
$e^{\frac{2\pi i m_{\lambda} }{l}}I$ with $0< m_{\lambda}\leq l$
and $I$ the identity matrix. As shown in Remark \ref{lambdabasis},
we have the well-defined complex
$\Lambda_G(\sigma)-$representations
$V_{\lambda})_{\sigma}:=V_{\lambda}\odot_{\mathbb{C}}
q^{\frac{m_{\lambda}}{l}}$ and
\begin{equation}(V)_{\sigma}:=\bigoplus_{\lambda}V_{\lambda}\odot_{\mathbb{C}}
q^{\frac{m_{\lambda}}{l}}\label{vgrepresentation}\end{equation}
Each $(V_{\lambda})_{\sigma}$ is the isotypic component of
$(V)_{\sigma}$ corresponding to the irreducible representation
$\lambda\odot_{\mathbb{C}} q^{\frac{m}{l}}$.

\begin{proposition}
Let $V$ be a faithful $G-$representation. And let $\sigma\in
G^{tors}$.

(i) If $V$ contains a trivial subrepresentation, $(V)_{\sigma}$ is
a faithful $\Lambda_G(\sigma)-$representation.

(ii) $(V)_{\sigma}\oplus (V)_{\sigma}\otimes_{\mathbb{C}}q^{-1}$
is a faithful $\Lambda_G(\sigma)-$representation.

(iii) $(V)_{\sigma}\oplus V^{\sigma}$ is a faithful
$\Lambda_G(\sigma)-$representation.

\label{farithlambda}
\end{proposition}

\begin{proof}

(i) Let $[a, t]\in \Lambda_G(\sigma)$ be an element acting
trivially on $(V)_{\sigma}$. Assume $t\in [0, 1)$. On
$(V_1)_{\sigma},$ $[a, t]v_0=e^{2\pi i t}v_0=v_0$. So $t=0$. Then
on the whole space $V_{\sigma}$, since $C_G(\sigma)$ acts
faithfully on it and for any $v\in V_{\sigma}$, $[a, 0]\cdot
v=a\cdot v=v$, then $a=e$.

So $(V)_{\sigma}$ is a faithful
$\Lambda_G(\sigma)-$representation.

(ii) Let $[a, t]\in \Lambda_G(\sigma)$ be an element acting
trivially on $V_{\sigma}$. Consider the subrepresentations
$(V_{\lambda})_{\sigma}$ and
$(V_{\lambda})_{\sigma}\otimes_{\mathbb{C}}q^{-1}$ of
$(V)_{\sigma}\oplus (V)_{\sigma}\otimes_{\mathbb{C}}q^{-1}$
respectively.  Let $v$ be an element in the underlying vector
space  $V_{\lambda}$. On $(V_{\lambda})_{\sigma}$, $[a, t]\cdot
v=e^{\frac{2\pi i m_{\lambda}t}{l}}a\cdot v= v$; and on
$(V_{\lambda})_{\sigma}\otimes_{\mathbb{C}}q^{-1}$, $[a, t]\cdot
v=e^{\frac{2\pi i m_{\lambda}t}{l}-2\pi i t}a\cdot v=v$. So we get
$e^{2\pi it}\cdot v=v$. Thus, $t=0$. $C_G(\sigma)$ acts faithfully
on $V$, so it acts faithfully on $(V)_{\sigma}\oplus
(V)_{\sigma}\otimes_{\mathbb{C}}q^{-1}$. Since $[a, 0]\cdot w=w$,
for any $w\in (V)_{\sigma}\oplus
(V)_{\sigma}\otimes_{\mathbb{C}}q^{-1}$, so $a=e$.

Thus, $(V)_{\sigma}\oplus (V)_{\sigma}\otimes_{\mathbb{C}}q^{-1}$
is a faithful $\Lambda_G(\sigma)-$representation.

(iii)  Note that $V^{\sigma}$ with the trivial $\mathbb{R}-$action
is the representation
$(V^{\sigma})_{\sigma}\otimes_{\mathbb{C}}q^{-1}$. The
representation $(V)_{\sigma}\oplus V^{\sigma}$ contains a
subrepresentation $(V^{\sigma})_{\sigma}\oplus
(V^{\sigma})_{\sigma}\otimes_{\mathbb{C}}q^{-1}$, which is a
faithful $\Lambda_G(\sigma)-$representation by the second
conclusion of Proposition \ref{farithlambda}. So
$(V)_{\sigma}\oplus V^{\sigma}$ is faithful.\end{proof}

\begin{lemma}
For any $\sigma\in G^{tors}$,  $(-)_{\sigma}$ defined in
(\ref{vgrepresentation}) is a functor from the category of
$G-$spaces to the category of $\Lambda_G(\sigma)-$spaces.
Moreover, $(-)_{\sigma}\oplus (-)_{\sigma}\otimes_{\mathbb{C}}
q^{-1}$ and $(-)_{\sigma}\oplus (-)^{\sigma}$ in Proposition
\ref{farithlambda} are also well-defined functors from the
category of $G-$spaces to the category of
$\Lambda_G(\sigma)-$spaces.
\end{lemma}

\begin{proof}
Let $f: V\longrightarrow W$ be a $G-$equivariant map. Then $f$ is
$C_G(\sigma)-$equivairant for each $\sigma\in G^{tors}$. For each
irreducible complex $C_G(\sigma)-$representation $\lambda$, $f:
V_{\lambda}\longrightarrow W_{\lambda}$ is
$C_G(\sigma)-$equivairant. And $f_{\sigma}:
(V_{\lambda})_{\sigma}\longrightarrow
(W_{\lambda})_{\sigma},\mbox{  } v\mapsto f(v)$ with the same
underlying spaces is well-defined and is
$\Lambda_G(\sigma)-$equivariant. It is straightforward to check if
we have two $G-$equivariant maps $f: V\longrightarrow W$ and $g:
U\longrightarrow V$, then $(f\circ g)_{\sigma}=f_{\sigma}\circ
g_{\sigma}.$ So $(-)_{\sigma}$ gives a well-defined functor from
the category of $G-$representations to the category of
$\Lambda_G(\sigma)-$representation.

The other conclusions can be proved in a similar way.\end{proof}
\begin{proposition} Let $H$ and $G$ be two compact Lie groups. Let $\sigma \in G$ and $\tau\in H$. Let $V$ be a $G-$representation and $W$ a $H-$representation.

(i) We have the isomorphisms of representations $(V\oplus
W)_{(\sigma, \tau)}=(V_{\sigma}\oplus
                     W_{\tau})$ as $\Lambda_{G\times H}(\sigma, \tau)\cong
                     \Lambda_G(\sigma)\times_{\mathbb{T}}\Lambda_H(\tau)-$representations;\\
$(V\oplus W)_{(\sigma, \tau)}\oplus (V\oplus W)_{(\sigma,
\tau)}\otimes_{\mathbb{C}}q^{-1}=((V)_{\sigma}\oplus
                    (V)_{\sigma}\otimes_{\mathbb{C}}q^{-1})\oplus ((W)_{\tau}\oplus
                     (W)_{\tau}\otimes_{\mathbb{C}}q^{-1})$ as $\Lambda_{G\times H}(\sigma, \tau)\cong
                     \Lambda_G(\sigma)\times_{\mathbb{T}}\Lambda_H(\tau)-$representations;\\
and  $(V\oplus W)_{(\sigma, \tau)}\oplus (V\oplus W)^{(\sigma,
\tau)}=((V)_{\sigma}\oplus V^{\sigma})\oplus ((W)_{\tau}\oplus
W^{\tau})$ as $\Lambda_{G\times H}(\sigma, \tau)\cong
                     \Lambda_G(\sigma)\times_{\mathbb{T}}\Lambda_H(\tau)-$representations.

(ii) Let $\phi: H\longrightarrow G$ be a group homomorphism. Let
$\phi_{\tau}: \Lambda_H(\tau)\longrightarrow
\Lambda_G(\phi(\tau))$ denote the group homomorphism obtained from
$\phi$. Then we have
$$\phi_{\tau}^*(V)_{\phi(\tau)}=(V)_{\tau},$$
$$\phi_{\tau}^*((V)_{\phi(\tau)}\oplus
(V)_{\phi(\tau)}\otimes_{\mathbb{C}}q^{-1})=(V)_{\tau}\oplus
(V)_{\tau}\otimes_{\mathbb{C}}q^{-1},$$
$$\phi_{\tau}^*((V)_{\phi(\tau)}\oplus V^{\phi(\tau)})=(V)_{\tau}\oplus V^{\tau}$$ as
                    $\Lambda_H(\tau)-$representations.
\end{proposition}
\begin{proof}
(i) Let $\{\lambda_G\}\mbox{    and    }\{\lambda_H\}$ denote the
sets of all the irreducible $C_G(\sigma)-$\\ representations and
all the irreducible $C_H(\tau)-$representations. Then $\lambda_G$
and $\lambda_H$ are irreducible representations of $C_{G\times
H}(\sigma, \tau)$ via the inclusion $C_G(\sigma)\longrightarrow
C_{G\times H}(\sigma, \tau)$ and $C_H(\tau)\longrightarrow
C_{G\times H}(\sigma, \tau)$.

The $\mathbb{R}-$representation assigned to each $C_{G\times
H}(\sigma, \tau)-$irreducible representation in $V\oplus W$ is the
same as that assigned to the irreducible representations of $V$
and $W$. So we have $$(V\oplus W)_{(\sigma,
\tau)}=(V_{\sigma}\oplus
                     W_{\tau})$$ as $\Lambda_{G\times H}(\sigma, \tau)\cong
                     \Lambda_G(\sigma)\times_{\mathbb{T}}\Lambda_H(\tau)-$representations.

Similarly we can prove the other two conclusions in (i).

(ii) Let $\sigma=\phi(\tau)$. If $(\phi_{\tau}^*V)_{\lambda_H}$ is
a $C_H(\tau)-$subrepresentation of $\phi_{\tau}^*V_{\lambda_G}$,
the $\mathbb{R}-$representation assigned to it is the same as that
to $V_{\lambda_G}$. So we have
$\phi_{\tau}^*(V)_{\phi(\tau)}=(V)_{\tau}$ as
$\Lambda_H(\tau)-$representations.

Similarly we can prove the other two conclusions in
(ii).\end{proof}

\subsection{real $\Lambda_G(\sigma)-$representation}\label{realdef}

In this section  we discuss real $\Lambda_G(\sigma)-$\\
representation and its relation with the complex
$\Lambda_G(\sigma)-$representations introduced in Lemma \ref{cl}.
The main reference is \cite{BT} and \cite{FH}.

Let $G$ be a compact Lie group, $\sigma\in G^{tors}$.
\begin{definition}A complex representation $\rho: G\longrightarrow Aut_{\mathbb{C}}(V)$ is said to be self dual if it is isomorphic to its complex dual
$\rho^*: G\longrightarrow Aut_{\mathbb{C}}(V^*)$ where
$V^*:=Hom_{\mathbb{C}}(V, \mathbb{C})$ and
$\rho^*(g)=\rho(g^{-1})^*$. \end{definition}

\begin{example}\label{rere1}
Let $\rho: C_G(g)\longrightarrow Aut_{\mathbb{R}}(V)$ be an
irreducible complex $C_G(g)-$\\ representation. Then as in Lemma
\ref{cl}, there exists a character $\eta:
\mathbb{R}\longrightarrow \mathbb{C}$ such that
$\rho(g)=\eta(1)I$. And $\rho\odot_{\mathbb{C}} \eta$ is an
irreducible complex representation of $\Lambda_G(g)$. Since
$(\rho\odot_{\mathbb{C}} \eta)^*([\alpha,
t])=\rho\odot_{\mathbb{C}} \eta([\alpha^{-1},
-t])^T=\rho(\alpha^{-1})^T\eta(-t)$, it is not self-dual if $\eta$
is nontrivial. In this case it is of complex type. And
$(V\odot_{\mathbb{C}}\eta)\oplus (V\odot_{\mathbb{C}}\eta)^*$ has
irreducible real form.

If $V$ is of real type, it is the complexification of a real
$C_G(g)-$representation $W$. If $g=e$ and the character $\eta$ we
choose is trivial, $(\rho\odot_{\mathbb{C}} \eta)^*([\alpha,
t])=\rho\odot_{\mathbb{C}} \eta([\alpha^{-1},
-t])^T=\rho(\alpha^{-1})^T\eta(-t)=\rho(\alpha^{-1})^T=\rho(\alpha)=(\rho\odot_{\mathbb{C}}
\eta)[\alpha, t]$ since $V$ is self-dual. In this case $W$ is a
real $\Lambda_G(g)-$representation via $[\alpha, t]\cdot w=\alpha
w$. And $V\odot_{\mathbb{C}}\eta$ is of real type since it is the
complexification of $W$. For  any nontrivial element $g$ in
$G^{tors}$, the $\Lambda_G(g)-$representation
$V\odot_{\mathbb{C}}\eta$  is of complex type, then
$(V\odot_{\mathbb{C}}\eta)\oplus (V\odot_{\mathbb{C}}\eta)^*$ is
of the real type.

If $V$ is of quaternion type, then $V=U_{\mathbb{C}}$ can be
obtained from a quaternion $C_G(g)-$representation $U$ by
restricting the scalar to $\mathbb{C}$.  If $g=e$ and $\eta$ is
trivial, $(\rho\odot_{\mathbb{C}} \eta)^*([\alpha,
t])=\rho\odot_{\mathbb{C}} \eta([\alpha^{-1},
-t])^T=\rho(\alpha^{-1})^T\eta(-t)=\rho(\alpha^{-1})^T=\rho(\alpha)=(\rho\odot_{\mathbb{C}}
\eta)[\alpha, t]$ since $V$ is self-dual.  In this case $W$ is a
quaternion $\Lambda_G(g)-$representation with $[\alpha, t]\cdot
w=\alpha w$. So $V\odot_{\mathbb{C}}\eta$ is of quaternion type.

Consider the case that $V$ is of complex type.  If $g=e$ and
$\eta$ is trivial, $(\rho\odot_{\mathbb{C}} \eta)^*([\alpha,
t])=\rho\odot_{\mathbb{C}} \eta([\alpha^{-1},
-t])^T=\rho(\alpha^{-1})^T\eta(-t)=\rho(\alpha^{-1})^T=\rho(\alpha)\neq(\rho\odot_{\mathbb{C}}
\eta)[\alpha, t]$ since $V$ is not self-dual.  So
$V\odot_{\mathbb{C}}\eta$ is of complex type.
\end{example}
For any compact Lie group, we use $RO(G)$ denote the real
representation ring of $G$. In light of the analysis in Example
\ref{rere1}, we have the following conclusion.
\begin{lemma}Let $\sigma\in G^{tors}$. Then the map $\pi^*: RO\mathbb{T}\longrightarrow RO\Lambda_G(\sigma)$
exhibits $RO\Lambda_G(\sigma)$ as a free $RO\mathbb{T}-$module.

In particular there is an $RO\mathbb{T}-$basis of
$RO\Lambda_G(\sigma)$ given by irreducible real representations
$\{V_{\Lambda}\}$. There is a bijection between $\{V_{\Lambda}\}$
and the set $\{\lambda\}$ of irreducible real representations of
$C_G(\sigma)$. When $\sigma$ is trivial, $V_{\Lambda}$ has the
same underlying space $V$ as $\lambda$. When $\sigma$ is
nontrivial,
$V_{\Lambda}=((\lambda\otimes_{\mathbb{R}}\mathbb{C})\odot_{\mathbb{C}}\eta)\oplus
((\lambda\otimes_{\mathbb{R}}\mathbb{C})\odot_{\mathbb{C}}\eta)^*$
where $\eta$ is a complex $\mathbb{R}-$representation such that
$(\lambda\otimes_{\mathbb{R}}\mathbb{C})(\sigma)$ acts on
$V\otimes_{\mathbb{R}}\mathbb{C}$ via the scalar multiplication by
$\eta(1)$. The dimension of $V_{\Lambda}$ is twice as that of
$\lambda$.\end{lemma}
As in (\ref{vgrepresentation}), we can construct a functor
$(-)^{\mathbb{R}}_{\sigma}$ from the category of real $G-$\\
representations to the category of real
$\Lambda_G(\sigma)-$representations with
\begin{equation}(V)^{\mathbb{R}}_{\sigma} =
(V\otimes_{\mathbb{R}}\mathbb{C})_{\sigma}\oplus
(V\otimes_{\mathbb{R}}\mathbb{C})^{*}_{\sigma}.\label{realfaithv}\end{equation}
\begin{proposition}
Let $V$ be a faithful real $G-$representation. And let $\sigma\in
G^{tors}$ and $l$ denote its order.  Then
$(V)^{\mathbb{R}}_{\sigma}$ is a faithful real
$\Lambda_G(\sigma)-$representation. \label{farithreallambda}
\end{proposition}
\begin{proof}
Let $[a, t]\in \Lambda_G(\sigma)$ be an element acting trivially
on $(V)^{\mathbb{R}}_{\sigma}$. Assume $t\in [0, 1)$. Let $v\in
(V\otimes_{\mathbb{R}}\mathbb{C})_{\sigma}$ and let $v^*$ denote
its correspondence in
$(V\otimes_{\mathbb{R}}\mathbb{C})^*_{\sigma}$. Then $[a, t]\cdot
(v+v^*)=(a e^{2\pi i mt}+a e^{-2\pi i m t})(v+v^*)= v+v^*$ where
$0< m \leq l$ is determined by $\sigma$. Thus  $a$ is equal to
both $e^{2\pi i mt} I$, and $e^{-2\pi i mt}I$. Thus $t=0$ and $a$
is trivial.

So $(V)^{\mathbb{R}}_{\sigma}$ is a faithful real
$\Lambda_G(\sigma)-$representation.\end{proof}
\begin{proposition} Let $H$ and $G$ be two compact Lie groups. Let $\sigma \in G^{tors}$ and $\tau\in H^{tors}$. Let $V$ be a real $G-$representation and $W$ a real $H-$representation.

(i) We have the isomorphisms of representations $(V\oplus
W)^{\mathbb{R}}_{(\sigma, \tau)}=(V^{\mathbb{R}}_{\sigma}\oplus
                     W^{\mathbb{R}}_{\tau})$ as $\Lambda_{G\times H}(\sigma, \tau)\cong
                     \Lambda_G(\sigma)\times_{\mathbb{T}}\Lambda_H(\tau)-$representations.

(ii) Let $\phi: H\longrightarrow G$ be a group homomorphism. Let
$\phi_{\tau}: \Lambda_H(\tau)\longrightarrow
\Lambda_G(\phi(\tau))$ denote the group homomorphism obtained from
$\phi$. Then
$\phi_{\tau}^*(V)^{\mathbb{R}}_{\phi(\tau)}=(V)^{\mathbb{R}}_{\tau},$
as $\Lambda_H(\tau)-$representations.\end{proposition} The proof
is left to the readers.

\section{Tedious proofs}\label{tediousproof}
\subsection{The proof of Lemma \ref{etaEcontinuous}}\label{etaEcontinuousProof}
\begin{proof}
When $v_1$ is infinity, $\eta_g(G, V)(v_1)$ is the basepoint of
$F_g(V)$. So by the construction of $QE_g(G, V)$ in Proposition
\ref{QEinal}, $v=v_1\wedge v_2$
 is mapped to the basepoint of $QE_g(G, V)$.
When $v_2$ is infinity, $\eta^{QE}_g(G, V)(v)$ is the basepoint.
So $\eta^{QE}_g(G, V)$ is well-defined. And since $\eta_g(G, V)$
is $C_G(g)-$equivariant, $\eta^{QE}_g(G, V)$ is
$C_G(g)-$equivariant.

Next we prove $\eta^{QE}_g(G, V)$ is continuous by showing for
each point in $QE_g(G, V)$, there is an open neighborhood of it
whose preimage is open in $S^V$. Consider a point $x$ in the image
of $\eta^{QE}_g(G, V)$ represented by $t_1a+t_2b$.

\textbf{Case I: } $0<t_2<1$ and $a$ is not the basepoint of
$F_g(G, V)$.

Let $A$ be an open neighborhood of $a$ in $F_g(G, V)$ not
including the basepoint. We can find such an $A$ since $F_g(G, V)$
is Hausdorff. Let $\delta>0$ be a small enough value. Let $U_{x,
\delta}$ be the open neighborhood of $x$
$$U_{x, \delta}:=\{[s_1\alpha + s_2\beta]\in QE_g(G, V)| \alpha\in A,
|s_2-t_2|<\delta, \|\beta-b\|<\delta\}.$$

Then $\eta^{QE}_g(G, V)^{-1}(U_{x, \delta})$ is the smash product
of $\eta_g(G, V)^{-1}(A), $ which is open in $S^{V^g}$, and an
open subset of $S^{(V^g)^{\perp}}$
$$\{w\in S^{(V^g)^{\perp}}|
t_2-\delta<\|w\|<t_2+\delta, \|w-b\|<\delta\}.$$ So it's open in
$S^V$.

\textbf{Case II:  } $t_2=0$ and $a$ is not the basepoint of
$F_g(G, V)$.

Let $A$ be an open neighborhood of $a$ in $F_g(G, V)$ not
including the basepoint. Let $\delta>0$ be a small enough value.
Let $W_{x, \delta}$ be the open neighborhood of $x$
$$W_{x, \delta}:=\{[s_1\alpha + s_2\beta]\in E_g(G, V)| \alpha\in A,
|s_2|<\delta, \|\beta-b\|<\delta\}.$$

Then $\eta^{QE}_g(G, V)^{-1}(W_{x, \delta})$ is the smash product
of $\eta_g(G, V)^{-1}(A), $ which is open in $S^{V^g}$, and an
open subset of $S^{(V^g)^{\perp}}$
$$\{w\in S^{(V^g)^{\perp}}|
\|w\|<\delta, \|w-b\|<\delta\}.$$ So it's open in $S^V$.

\textbf{Case III:   } $x$ is the basepoint $x_g$ of $QE_g(G, V)$.

Let $A_0$ be an open neighborhood of the basepoint $c_0$.
For any point $w$ of the form $t_1c_0+t_2b$ in the the space
$QE'_g(G, V)$ with $0<t_2<1$,  let $U_{w, \delta_w}$ denote the
open subset of $QE_g(G, V)$
$$\{[s_1\alpha + s_2\beta]\in QE_g(G, V)| \alpha\in A_0,
|s_2-t_2|<\delta_w, \|\beta-b\|<\delta_w\}$$ with $\delta_w$ small
enough.
Let $W_{\delta}$ denote the open subset of $QE_g(G, V)$
 $$\{[s_1\alpha + s_2\beta]\in QE_g(G, V)| \alpha\in A_0, |s_2|<\delta,
\|\beta-b\|<\delta\}$$ with $\delta$ small enough.


For any $b\in S(G, V)_g$ with $\|b\|\leqslant 1$, let $V_{b,
\delta_b}$ denote the open subset of $QE_g(G, V)$
$$\{[s_1\alpha + s_2\beta]\in QE_g(G, V)| 
s_2>1-\delta_b, \|\beta-b\|<\delta_b\}$$ with $\delta_b$ small
enough.

We consider the open neighborhood $U$ of $x$ that is the union of
the spaces defined above
$$U:=(\bigcup_{w} U_{w, \delta_w})\cup W_{\delta}\cup
(\bigcup_{b} V_{b, \delta_b})$$ where $w$ goes over all the points
of the form $[t_1c_0+t_2b]$ in $QE_g(G, V)$ with $0<t_2<1$,  and
$b$ goes over all the points in $S(G, V)_g$ with $\|b\|\leqslant
1$.

The preimage of each $U_{w, \delta_w}$ and $W_{\delta}$ is open,
the proof of which is analogous to Case I and II. The preimage of
$V_{b, \delta_b}$ is the smash product of $S^{V^g}$ and the open
set of $S^{(V^g)^{\perp}}$ $$\{w_2\in S^{(V^g)^{\perp}} |
\|w_2\|>1-\delta_b, \|w_2-b\|<\delta_b\},$$ thus, is open.

The preimage of $U$  is the union of open subsets in $S^V$, thus,
open.

Therefore, The map $\eta^{QE}_g(G, V)$ defined in (\ref{finaleta})
is continuous.\end{proof}

\subsection{The proof of Lemma \ref{muEcontinuous}}
\label{muEcontinuousProof}

\begin{proof}

Note that when either $a_1$ is the basepoint of $F_g(G, V)$, or
$a_2$ is the basepoint of $F_h(H, W)$, or $t_2=1$, or $u_2=1$, the
point $[t_1a_1+  t_2b_1]\wedge [u_1a_2 + u_2b_2]$ is mapped to the
basepoint $x_{g, h}$. The spaces $S(G, V)_g$ have the following
properties:

(i) There is no zero vector in any $S(G, V)_g$ by its
construction;

(ii)  For any $b_1\in S(G, V)_g$, $b_2\in S(H, W)_h$, $b_1$, $b_2$
and $b_1+b_2$ are all in $S(G\times H, V\oplus W)_{(g, h)}.$ $b_1$
and $b_2$ are orthogonal to each other, so
$\|b_1+b_2\|^2=\|b_1\|^2+\|b_2\|^2$. Thus,  if $t_2u_2\neq 0$,
$\|b_1+b_2\|\leqslant \sqrt{t_1^2+t_2^2}$.

Therefore, $\mu^{QE}_{(g, h)}((G,V), (H, W))$ is well-defined.

Let $x=[s_1\alpha+s_2\beta]$ be a point in the image of
$\mu^{QE}_{(g, h)}((G,V), (H, W))$. If $s_2$ is nonzero, there is
unique $\beta_1\in S(G, V)_g\cup\{0\}$ and unique $\beta_2\in S(H,
W)_h\cup\{0\}$ such that $\beta=\beta_1+\beta_2$.

For each point in the image, we pick an open neighborhood of it so
that its preimage in $QE_g(G, V)\wedge QE_h(H, W)$ is open.

\textbf{Case I: } $x$ is not the basepoint,  $0<s_1, s_2<1$ and
$\beta_1$ and $\beta_2$ are both nonzero.


Let $A(\alpha)$ be an open neighborhood  of $\alpha$ in $F_{(g,
h)}(G\times H, V\oplus W)$ not containing the basepoint. Let
$\delta>0$ be some small enough value. We consider the open
neighborhood $U_{x, \delta}$ of $x$
$$U_{x, \delta}:= \{[r_1a+r_2d]\mbox{   }|\mbox{  } \|d_1-\beta_1\|<\delta, \|d_2-\beta_2\|<\delta, a\in
A(\alpha), |r_2^2-s_2^2|<\delta \}$$ where  
$d=d_1+d_2$ with $d_1\in S(G, V)_g\cup\{0\}$ and $d_2\in S(H,
W)_h\cup\{0\}$.


The preimage of $U_{x, \delta}$ is
\begin{align*}\{[t_1a_1+t_2d_1]\wedge [u_1a_2+ u_2d_2]\mbox{  }| \mbox{   }&a_1\wedge a_2\in \mu^F_{(g, h)}((G, V), (H, W))^{-1}(A(\alpha)),\\
&\|d_1-\beta_1\|<\delta, \|d_2-\beta_2\|<\delta, |t_2^2+
u^2_2-s_2^2|<\delta \},\end{align*} where $\mu^F_{(g, h)}((G, V),
(H, W))$ is the multiplication defined in (\ref{muFKg}).

Note that $QE_g(G, V)\wedge QE_h(H, W)$ is the quotient space of a
subspace of the product of spaces
$$F_g(G, V)\times S(G, V)_g\times [0, 1]\times F_h(H, W)\times
S(H, W)_h\times [0, 1]$$  and  $U_{x, \delta}$ is the quotient of
an open subset of this product. So it is open in $QE_g(G, V)\wedge
QE_h(H, W)$.

\textbf{Case II:   } $x$ is not the basepoint,  $0<s_1, s_2<1$ and
$\beta\in S(H, W)_{h}$.

Let $A(\alpha)$ be an open neighborhood of $\alpha$ in $F_{(g,
h)}(G\times H, V\oplus W)$ not containing the basepoint. Let
$\delta>0$ be some small enough value. Consider the open
neighborhood $W_{x, \delta}$ of $x$ $$W_{x,
\delta}:=\{[r_1a+r_2d]\mbox{  }|\mbox{   } \|d_1-\beta_1\|<\delta,
\|d_2\|<\delta, a\in A(\alpha), |r^2_2-s^2_2|<\delta\}$$ where
$d=d_1+d_2$ with $d_1\in S(G, V)_g\cup\{0\}$ and $d_2\in S(H,
W)_h\cup\{0\}$.

The preimage of $W_{x, \delta}$ is
\begin{align*}\{[t_1a_1+t_2d_1]\wedge [u_1a_2+ u_2d_2]\mbox{   }|\mbox{   } &a_1\wedge a_2\in \mu^F_{(g, h)}((G, V), (H, W))^{-1}(A(\alpha)),\\
&\|d_1-\beta_1\|<\delta, \|d_2\|<\delta, |t_2^2+
u^2_2-s_2^2|<\delta \}.\end{align*}

It is the quotient of an open subspace of the product $$F_g(G,
V)\times S(G, V)_g\times [0, 1]\times F_h(H, W)\times S(H,
W)_h\times [0, 1].$$ So the preimage of $W_{x, \delta}$ is open in
$QE_g(G, V)\wedge QE_h(H, W)$.

\textbf{Case III:   }  $x$ is not the basepoint,  $0<s_1, s_2<1$
and  $\beta\in S(G, V)_g$. We can show the map is continuous at
such points in a way analogous to Case II.

\textbf{Case IV} $x$ is not the basepoint and $s_2$ is zero.

Let $A(\alpha)$ be an open neighborhood of $\alpha$ in $F_{(g,
h)}(G\times H, V\oplus W)$ not containing the basepoint. Let
$\delta>0$ be some small enough value.

Consider the open neighborhood of $x$ $$B_{x,
\delta}:=\{[r_1a+r_2d]\mbox{  }|\mbox{   }  a\in A(\alpha),
\|d_1\|<\delta, \|d_2\|<\delta,  0\leqslant r_2^2<\delta\}$$ where
$d=d_1+d_2$ with $d_1\in S(G, V)_g\cup\{0\}$ and $d_2\in S(H,
W)_h\cup\{0\}$.

The preimage of $B_{x, \delta}$ is
\begin{align*}\{[t_1a_1+t_2d_1]\wedge [u_1a_2+ u_2d_2]\mbox{   }|\mbox{  } &a_1\wedge a_2\in \mu^F_{(g, h)}((G, V), (H, W))^{-1}(A(\alpha)),\\
&\|d_1\|<\delta, \|d_2\|<\delta, 0\leqslant t_2^2+ u^2_2<\delta
\}.\end{align*}

It is the quotient of an open subspace of the product $$F_g(G,
V)\times S(G, V)_g\times [0, 1]\times F_h(H, W)\times S(H,
W)_h\times [0, 1].$$ So the preimage of $B_{x, \delta}$ is open in
$QE_g(G, V)\wedge QE_h(H, W)$.

\textbf{Case V:   }$x=[s_1\alpha+s_2\beta]$ is the base point.

Let $A_0(\alpha)$ be an open  neighborhood of $\alpha$ in $F_{(g,
h)}(G\times H, V\oplus W)$.
For any point $w$ in $QE'_{(g, h)}(G\times H, V\oplus W)$ of the
form $t_1c_0+t_2b$ with $0<t_2<1$ and $b_1$ $b_2$ both nonzero,
let $U_{w, \delta_w}$ be the open subset of $QE_{(g, h)}(G\times
H, V\oplus W)$
$$\{[r_1a+r_2d]\mbox{   }|\mbox{   }
\|d_1-b_1\|<\delta_w, \|d_2-b_2\|<\delta_w, a\in A_0(\alpha),
|r_2^2-t_2^2|<\delta_w \}$$
 with $\delta_w$ small
enough. 

For each point $y$  in $QE'_{(g, h)}(G\times H, V\oplus W)$  of
the form $t_1c_0+t_2b$ with $0<t_2<1$ and $b\in S(H, W)_h$, let
$W_{y, \delta_y}$ be the open subset
 $$\{[r_1a+r_2d]\in QE_{(g, h)}(G\times H, V\oplus
W)| \|d_1-b_1\|<\delta_y, \|d_2\|<\delta_y, a\in A_0(\alpha),
|r^2_2-t^2_2|<\delta_y\}$$ with $\delta_y$ small enough. 
For each point  $z$ in $QE'_{(g, h)}(G\times H, V\oplus W)$ of the
form $t_1c_0+t_2b$ with $0<t_2<1$ and $b\in S(G, V)_g$, let $V_{z,
\delta_z}$ be the open subset $$\{[r_1a+r_2d]\in QE_{(g,
h)}(G\times H, V\oplus W)| \|d_2-b_2\|<\delta_z, \|d_1\|<\delta_z,
a\in A_0(\alpha), |r^2_2-t^2_2|<\delta_z\}$$ with $\delta_z$ small
enough. 
Let $B_{x_0, \delta}$ denote the open set $$\{[r_1a+r_2d]\in
QE_{(g, h)}(G\times H, V\oplus W)| \|d_2\|<\delta, \|d_1\|<\delta,
a\in A_0(c_0), 0 \leqslant r_2<\delta\}$$ with $\delta$ small
enough. For each $\theta$ in $QE'_{(g, h)}(G\times H, V\oplus W)$
of the form $0+ 1b$, let $D_{\theta, \delta_{\theta}}$ be the open
subset
$$\{[r_1a+r_2d]\in QE_{(g, h)}(G\times H, V\oplus W)|
\|d-b\|<\delta_{\theta}, 1\geqslant r_2\geqslant
1-\delta_{\theta}\}$$ with $\delta_{\theta}$ small enough. 

Then we consider the open neighborhood of $x$ in $QE_{(g,
h)}(G\times H, V\oplus W)$ that is the union of the spaces above
$$U:=(\bigcup_{w} U_{w, \delta_w})\cup (\bigcup_y W_{y,
\delta_y})\cup (\bigcup_{z} V_{z, \delta_z}) \cup B_{x_0, \delta}
\cup (\bigcup_{\theta} D_{\theta, \delta_{\theta}})$$ where $w$
goes over all the points  in $QE'_{(g, h)}(G\times H, V\oplus W)$
of the form $t_1c_0+t_2b$ with $0<t_2<1$ and $b_1$, $b_2$ both
nonzero,
 $y$ goes over all the points in $QE'_{(g, h)}(G\times H, V\oplus W)$ of the form $t_1c_0+t_2b$ with $0<t_2<1$
and $b\in S(H, W)_h$,  $z$ goes over all the points in $QE'_{(g,
h)}(G\times H, V\oplus W)$ of the form $t_1c_0+t_2b$ with
$0<t_2<1$ and $b\in S(G, V)_g$, and $\theta$ goes over  all the
points of the form $0+ 1b$ in $QE'_{(g, h)}(G\times H, V\oplus
W)$.

The preimage of each $U_{x, \delta_x}$,  $W_{y, \delta_y}$, $V_{z,
\delta_z}$, $B_{x_0, \delta}$ is open, the proof of which are
analogous to that of Case I, II, III and IV. The preimage of
$D_{\theta, \delta_{\theta}}$ is
$$\{[t_1a_1+t_2d_1]\wedge [u_1a_2+ u_2d_2]\mbox{  }| \mbox{   } \|d_1+d_2-b\|<\delta_{\theta}, 1-\sqrt{ t_2^2+
u^2_2}<\delta_{\theta} \},$$ which is open. Therefore,  the
preimage of $U$ is open.

Combining all the cases above,  the multiplication $\mu^{QE}_{(g,
h)}((G,V), (H, W))$ defined in (\ref{muEg}) is
continuous.\end{proof}

\subsection{The proof of Lemma \ref{Gorthocommdiaglemma}} \label{GorthocommdiaglemmaProof}

\begin{proof}
In this proof, we identify the end $F_g(G, V)$ in the space $
QE_g(G, V)$ with the points of the form $(a, 0, 0)$, i.e. $1a+00$,
in the space (\ref{qimiaoanoano}) as indicated in Remark
\ref{ananotherEll}. If the coordinate $t_2$ in a point $t_1a+t_2b$
is zero, then $b$ is the zero vector.

(i) Unity.

Let $v\in S^V$ and $w\in S^W$. Let $v=v_1\wedge v_2\mbox{, with
         } v_1\in S^{V^g}\mbox{, and   }v_2\in S^{(V^g)^{\perp}},$
         and
$w=w_1\wedge w_2\mbox{, with            } w_1\in S^{W^h}\mbox{,
and       }w_2\in S^{(W^h)^{\perp}}.$
$$\mu^{QE}_{(g, h)}\big((G, V), (H, W)\big)\circ \big(\eta^{QE}_g(G,
V)\wedge \eta^{QE}_h(H, W)\big)(v\wedge w)$$ is the basepoint if
$\|v_2\|^2+\|w_2\|^2\geqslant 1$. If $\|v_2\|^2+\|w_2\|^2
\leqslant 1$, it equals
\begin{equation}\big[(1-\sqrt{\|v_2\|^2+\|w_2\|^2})\eta_g(G,
V)(v_1)\wedge\eta_h(H, W)(w_1)+
\sqrt{\|v_2\|^2+\|w_2\|^2}(v_2+w_2)\big]
.\label{Gorthocommdiaglemunity1}\end{equation} On the other
direction, $\eta^{QE}_{(g, h)}(G\times H, V\oplus W)(v\wedge w)$
is the basepoint if $\|v_2+w_2\|\geqslant 1$. Note that since
$v_2$ and $w_2$ are orthogonal to each other,
$\|v_2+w_2\|^2=\|v_2\|^2+\|w_2\|^2$.

If $\|v_2+w_2\|\leqslant 1$, it is
\begin{equation}[(1-\sqrt{\|v_2\|^2+\|w_2\|^2})\eta_{g}(G,
V)(v_1)\wedge\eta_h(H, W)(w_1)+
\sqrt{\|v_2\|^2+\|w_2\|^2}(v_2+w_2)]
,\label{Gorthocommdiaglemunity2}\end{equation} which is equal to
the term in (\ref{Gorthocommdiaglemunity1}) by Proposition
\ref{newF} (ii).

(ii) Associativity.

Let $x=[t_1a_1+t_2b_1]$ be a point in $QE_g(G, V)$,
$y=[s_1a_2+s_2b_2]$ a point in $QE_h(H, W)$, and
$z=[r_1a_3+r_2b_3]$ a point in $QE_k(K, U)$.
$$\mu^{QE}_{((g, h), k)}((G\times H, V\oplus W), (K, U))\circ
(\mu^{QE}_{(g, h)}((G, V), (H, W))\wedge Id)(x\wedge y\wedge z)$$
is the basepoint if $t_2^2+s_2^2+r_2^2\geqslant 1$.

If $t_2^2+s_2^2+r_2^2\leqslant 1$,
\begin{align*}&\mu^{QE}_{((g, h), k)}((G\times
H, V\oplus W), (K, U))\circ (\mu^{QE}_{(g, h)}((G, V), (H,
W))\wedge Id)(x\wedge y\wedge z) \\ = &\mu^{QE}_{((g, h),
k)}((G\times H, V\oplus W), (K, U))\\
&([(1-\sqrt{t_2^2+s_2^2})\mu^F_{g, h}((G,V)
, (H, W))(a_1\wedge a_2) + \sqrt{t^2_2+s^2_2}(b_1+b_2)]\wedge z)\\
=& [(1-\sqrt{t_2^2+s_2^2+r_2^2})\mu^F_{((g, h), k)}((G\times H,
V\oplus W), (K, U))(\mu^F_{g, h}((G,V) , (H, W))(a_1\wedge a_2)
\wedge a_3)\\ &+
\sqrt{t^2_2+s^2_2+r^2_2}(b_1+b_2+b_3)]\end{align*}
Then consider the other direction. $$\mu^{QE}_{(g, (h, k))}((G,
V), (H\times K, W\oplus U))\circ(Id\wedge\mu^{QE}_{(h, k)}(H\times
K, W\oplus U))(x\wedge y\wedge z)$$ is the basepoint if
$t_2^2+s_2^2+r_2^2\geqslant 1$. If $t_2^2+s_2^2+r_2^2\leqslant 1$,
\begin{align*} &\mu^{QE}_{(g, (h, k))}((G, V), (H\times K,
W\oplus U))\circ(Id\wedge\mu^{QE}_{(h, k)}(H\times K, W\oplus
U))(x\wedge y\wedge z)\\ =&\mu^{QE}_{(g, (h, k))}((G, V), (H\times
K, W\oplus U))\\ & (x\wedge [(1-\sqrt{r_2^2+s_2^2})\mu^F_{(h, k)}(
(H, W), (K, U))(a_2\wedge a_3) + \sqrt{r^2_2+s^2_2}(b_2+b_3)])\\
=& [(1-\sqrt{t_2^2+s_2^2+r_2^2})\mu^F_{(g, (h, k))}((G, V),
(H\times K, W\oplus U))(a_1\wedge \mu^F_{(h, k)}( (H, W),
(K, U))(a_2\wedge a_3) )\\
&+ \sqrt{t^2_2+s^2_2+r^2_2}(b_1+b_2+b_3)]\mbox{, which  by
Proposition }\ref{newF}\mbox{ (ii) is equal to }
\end{align*} \begin{align*}&
[(1-\sqrt{t_2^2+s_2^2+r_2^2})\mu^F_{((g, h), k)}((G\times H,
V\oplus W), (K, U))(\mu^F_{(g, h)}((G,V) , (H, W))(a_1\wedge a_2)
\wedge a_3)\\ &+
\sqrt{t^2_2+s^2_2+r^2_2}(b_1+b_2+b_3)].\end{align*}

(iii) Centrality of unit.

Let $v\in S^V$ and  $x=[t_1a+t_2b]$  a point in $QE_h(H, W)$.
$$QE_{(g, h)}(\tau)\circ \mu^{QE}_{(g, h)}((G,V), (H, W))\circ
(\eta^{QE}_g(G, V)\wedge Id)(v\wedge x)$$ is the base point if
$\|v_2\|^2+t_2^2\geqslant 1$. If $\|v_2\|^2+t_2^2\leqslant 1$,  by
Proposition \ref{newF} (ii) it is
\begin{align*}&[(1-\sqrt{\|v_2\|^2+ t_2^2})\mu^F_{(g, h)}((G,V),  (H,
W))(\eta_g(G, V)(v_1)\wedge a)+ \sqrt{\|v_2\|^2+ t_2^2}(v_2+b)]\\
=&[(1-\sqrt{\|v_2\|^2+ t_2^2})\mu^F_{(h, g)}((H, W), (G, V))(a
\wedge \eta_g(G, V)(v_1))+ \sqrt{\|v_2\|^2+
t_2^2}(v_2+b)].\end{align*}
$$\mu^{QE}_{(h, g)}((H, W), (G, V))\circ(Id\wedge\eta^{QE}_h(H, W)) \circ
\tau (v\wedge x)$$ is the base point if $\|v_2\|^2+t_2^2\geqslant
1$. If $\|v_2\|^2+t_2^2\leqslant 1$, it is
$$[(1-\sqrt{\|v_2\|^2+ t_2^2})\mu^F_{(h, g)}((H, W), (G, V))(a \wedge
\eta_g(G, V)(v_1))+ \sqrt{\|v_2\|^2+ t_2^2}(v_2+b)].$$


(iv) 
$\mu^{QE}_{(g, h)}((G, V), (H, W))(x\wedge y)=
[(1-\sqrt{t_2^2+s_2^2})\mu^F_{g, h}((G,V) , (H, W))(a_1\wedge a_2)
+ \sqrt{t^2_2+s^2_2}(b_1+b_2)]=[(1-\sqrt{t_2^2+s_2^2})\mu^F_{h,
g}( (H, W), (G, V))(a_2\wedge a_1) + \sqrt{t^2_2+s^2_2}(b_2+b_1)]=
\mu^{QE}_{(h, g)}((H, W), (G, V))(y\wedge x).$
\end{proof}

\end{document}